\documentclass[a4paper, 11pt, reqno]{amsart} %reqno for right equation number 
\usepackage[dvipsnames]{xcolor}
\usepackage[english]{babel}               % [french, frenchb, english, ]
\usepackage[T1]{fontenc} % fontenc is oriented to output, that is, what fonts to use for printing characters. 
\usepackage{lmodern} 
\usepackage[utf8]{inputenc} % inputenc allows the user to input accented characters directly from the keyboard; 
  \usepackage{amssymb}      % may be redundant with amsmath
\usepackage{amsthm}
\theoremstyle{plain}
\newtheorem{theorem}{Theorem}[section]

\newtheorem*{proposition*}{Proposition}
\newtheorem*{theorem*}{Theorem}
\newtheorem{lemma}[theorem]{Lemma}
\newtheorem*{lemma*}{Lemma}

\newtheorem{corollary}[theorem]{Corollary}
\theoremstyle{definition}

\newtheorem{remark}[theorem]{Remark}
\usepackage{mathrsfs}      % for the command \mathscr{}
 \usepackage{bm} % For better bolding of math symbols 
\usepackage{tikz} 
\usepackage{enumerate}    
\usepackage{mathtools} 
\usepackage{textcomp}  % provides additional text symbols, including currency signs, trademark symbols, copyright symbols, and other special characters for use in documents.
\mathtoolsset{showonlyrefs} %hides equation numbers for all equations except those that are referenced in the document. 
\numberwithin{equation}{section}  % reset equation counters at start of each "section" and prefix numbers by section number
 \usepackage[margin=3cm]{geometry}     % very customizable margins. Under some (rare) circumstances, should be loaded after hyperref
\usepackage[babel]{microtype} % many good lay-out/justification effects, see:
\usepackage[symbol]{footmisc}
 \usepackage{comment}      % provide new {comment} environment: all text inside the environment is ignored.
 \usepackage{appendix}
\usepackage{graphicx}  
\usepackage{float}                      % Improved interface for floating objects ; add [H] option

\usepackage[
breaklinks   = true,        % so long urls are correctly broken across lines
pagebackref  = true,     % add page number in bibliography and link to position in document where cited
]{hyperref} 
\hypersetup{
	colorlinks=true,
	pdfpagemode=UseNone,
    citecolor=Green,
    linkcolor=blue,
    urlcolor=black,
	pdfstartview=FitW
}
\def\N{\mathbb{N}}

\def\R{\mathbb{R}}
\def\E{\mathbf{E}}
\def\P{\mathbf{P}} 
\def\Var{\mathbf{Var}}

\newcommand{\ind}[1]{\mathbf{1}_{ \{ #1  \}}}    %Indicator function 

\DeclareMathOperator*{\argmin}{arg\,min}

\newcommand*{\dif}{\ensuremath{\mathop{}\!\mathrm{d}}}

 \renewcommand{\bar}[1]{\mkern 1mu\overline{\mkern-1mu #1\mkern-1mu}\mkern 1mu}
\renewcommand{\hat}[1]{\widehat{#1}}

\title[]{Minimum-Weight Path in a Sparse Erd\H{o}s--R\'{e}nyi Graph with Signed Weights}

\author[H. Ma]{Heng Ma} 

\address[Heng Ma]
{Faculty of Data and Decision Sciences, Technion - Israel's Institute of Technology, Haifa, 32000, Israel.}
\email{hengmamath at gmail dot com}
\urladdr{\url{https://hengmamath.github.io}}

\author[P. Maillard]{Pascal Maillard} 

\address[Pascal Maillard]{Institut de Mathématiques de Toulouse, CNRS UMR 5219, Université de Toulouse, 118 route de Narbonne, 31062 Toulouse Cedex 9, France and Institut Universitaire de France.}

\email{Pascal.Maillard at math dot univ-toulouse dot fr}
\urladdr{\url{https://www.math.univ-toulouse.fr/~pmaillar/}} 


\date{\today}

\begin{document}
 
\begin{abstract}
 We consider a sparse Erd\H{o}s--R\'{e}nyi graph $\mathcal{G}(n,\lambda/n)$ where each edge is independently assigned a random signed weight. For two uniformly chosen vertices, we study the joint distribution of the total weights and hopcounts (number of edges) of the near-minimum weight paths connecting them. Under certain conditions on the weight distribution, which ensure in particular that these paths are typically of positive weight, we prove that the point process formed by the rescaled pairs of total weight and hopcount, converges weakly to a Poisson point process with a random intensity. This random intensity is characterized by the product of two independent copies of the Biggins martingale limit of certain branching random walk. This result generalizes the work of Daly, Schulte, and Shneer (arXiv:2308.12149) from non-negative to signed weights.  
\end{abstract}

\maketitle

 %\tableofcontents 

\section{Introduction}
Consider a sparse Erd\H{o}s--R\'{e}nyi graph $\mathcal{G}_{n}=\mathcal{G}_{n, {\lambda}/{n}}$ on the vertex set $[n]\coloneq \{1,2,\dots,n\}$, where each edge exists independently with probability ${\lambda}/{n}$ for some fixed constant $\lambda > 1$. Each edge $e$ in $\mathcal{G}_{n}$ is assigned an i.i.d. real-valued weight $X(e) \in \mathbb{R}$, distributed according to a random variable $X$. Let $\psi_X(t) \coloneq  \ln \mathbf{E}[e^{-tX}] \in (-\infty, \infty]$, $t \in \mathbb{R}$ denote the log-Laplace transform of $X$ and denote by $\mathcal{D} = \{t\in\R: \psi_X(t) < \infty\}$ its domain of definition, which is an interval. In this paper, we investigate the minimum total weight of a simple path\footnote{Henceforth, we use path to mean a simple path.} from vertex $1$ to vertex $n$, denoted by: 
\begin{equation}
 \mathbb{X}^{*}_{n}  \coloneq  \min\left\{  X(p) : p \in \mathcal{P}(1,n) \right\} ,
\end{equation}
Here, $\mathcal{P}(1, n)$ denotes the set of all paths from $1$ to $n$ in $\mathcal{G}_{n}$, and $X(p) \coloneq \sum_{e \in p} X(e)$ represents the total weight of the path $p$. By convention, if vertices $1$ and $n$ are disconnected, i.e., $\mathcal{P}(1,n)=\emptyset$, we define $\min \emptyset =\infty$. Let $\mathcal{P}^{*}(1,n) = \{ p \in \mathcal{P}(1,n): X(p)= \mathbb{X}^{*}_{n} \}$ be the set of paths achieving this minimal total weight. We further define:
\begin{equation}
   \mathbb{H}^{*}_{n}   \coloneq   \min\left\{  H(p) : p \in \mathcal{P}^{*}(1,n)     \right\} .
\end{equation}
where $H(p)$ denotes the hopcount (number of edges) of the path $p$. Our primary focus is on the asymptotic behavior of the pair $(\mathbb{X}^*_{n}, \mathbb{H}^{*}_{n})$.

Throughout this paper, we require that there exists $\alpha\in(0,\infty) \cap \operatorname{int}\mathcal D  $ (the interior of the interval $\mathcal D$), such that
\begin{equation}\label{def-alpha}
 \lambda \mathbf{E}[ e^{-\alpha X} ]   =  \lambda e^{\psi_{X}(\alpha)}=  1,\quad\text{ and }\quad \psi_X'(\alpha) < 0.
\end{equation}
Also, for simplicity, we assume that the distribution of $X$ is non-arithmetic:
\begin{equation}\label{eq-assumption-4}
 \P( X \in  d \mathbb{Z} ) \neq 1  \ , \  \forall  \,  d >0. 
\end{equation}
We indicate in Remark~\ref{rmk-lattice-1} how to generalize our results to the arithmetic case.

For non-negative weights, i.e., $\P(X \ge 0)=1$,  the condition $\lambda \P(X = 0) < 1$  guarantees the existence of such an $\alpha$ in \eqref{def-alpha}. In that case, Daly, Schulte, and Shneer \cite{DSS23} proved that the scaled variables
$
 (  \mathbb{X}^{*}_{n} - \frac{1}{\alpha} \ln n, \frac{ 1}{\sqrt{\beta \ln n}}(\mathbb{H}^{*}_{n } - \gamma \ln n) )
$
 converge jointly in distribution to  a non-degenerate limit,  where   the constants $\gamma$ and $\beta$ are given by 
 \begin{equation}
  \gamma \coloneq   \frac{1}{\alpha |\psi'_{X}(\alpha)|} \ \text{ and } \ \beta\coloneq  \frac{\psi_{X}''(\alpha)}{\alpha |\psi_{X}'(\alpha)|^3} = \gamma \frac{\psi_{X}''(\alpha)}{|\psi_{X}'(\alpha)|^2}.
 \end{equation}  
Moreover, using the method of moments, they further characterized the full extremal statistics of the minimal weight.  They considered   the \emph{extremal process}
\begin{equation}
\mathcal{E} _{n} \coloneq  \sum_{p \in \mathcal{P}(1,n)} \delta_{(  X(p)- \frac{1}{\alpha}\ln n ,\frac{ H(p)- \gamma \ln n}{\sqrt{\beta \ln n}}   )}
\end{equation}
and showed that, as $n \to \infty$ it converges  in distribution (in the space of random measures on $\R^2$ endowed with the vague topology) to a limiting process $\mathcal{E}_\infty$. This process is a Cox process, i.e. a Poisson point process 
with a random intensity $W\tilde{W} \Lambda(\dif x,\dif h)$, where 
 \begin{equation}
  \Lambda(\dif x,\dif h) \coloneq  \gamma \cdot \Big( \alpha e^{\alpha x} \dif x  \Big) \otimes \Big( \frac{1}{\sqrt{2 \pi}} e^{-\frac{h^2}{2}}   \dif h \Big)
 \end{equation}
  and   $W$,  $\tilde{W}$ are i.i.d. copies of the Biggins martingale limit with parameter $\alpha$ of the branching random walk with  offspring distribution $\mathrm{Poi}(\lambda)$ and displacement distribution $X$, see \eqref{eq-Biggins-martingale-limit} for a precise definition.

Our main theorem extends these results to cases with signed weights. 
% For this, we require some additional assumptions on the weight distribution $X$. Let us define 
% \begin{equation}\label{def-kappa}
%   \kappa \coloneq  \sup\left\{ t> \alpha :     \psi_{X}(t) \le \psi_{X}(\alpha)\right\} \in (\alpha, \infty] .
% \end{equation}
% Note that $\kappa = \infty$ implies $\P(X \ge 0)=1$, since otherwise $\lim_{t\to \infty} \E[e^{-t X}]=\infty$. When $\kappa < \infty$, we assume that
% \begin{equation}\label{eq-assumption-3}
%   \kappa \in \operatorname{int}\mathcal D,
% \end{equation} 
% and this particularly implies that $ \psi_{X}(\kappa)= \psi_{X}(\alpha)$. Finally, we assume that the distribution of $X$ is non-arithmetic:
% \begin{equation}\label{eq-assumption-4}
%  \P( X \in  d \mathbb{Z} ) \neq 1  \ , \  \forall  \,  d >0. 
%\end{equation} 
We also slightly strengthen the notion of convergence, by considering 
$\mathcal{E}_\infty$ and $\mathcal{E}_n$, $n\in\N$ as random elements in the space of measures on $\R^{-\infty}\times \R^{-\infty,+\infty}$, where $\R^{-\infty} = \{-\infty\}\cup \R$ and $\R^{-\infty,+\infty} = \{-\infty\}\cup\R\cup\{+\infty\}$ are extensions of the real line endowed with the usual topology (in particular, the space $\R^{-\infty,\infty}$ is compact).

\begin{theorem}\label{thm1}
% Let $  \mathcal{E}_{\infty} \coloneq  \mathrm{PPP}( W \tilde{W}  \Lambda(\dif x \otimes \dif h))$ be a Position point process with random intensity $ W \tilde{W}  \Lambda(\dif x \otimes \dif h) $.   
Under assumptions \eqref{def-alpha}
 %, \eqref{def-kappa}, \eqref{eq-assumption-3}
 and \eqref{eq-assumption-4},  the process $\mathcal{E}_n$ converges in law to $\mathcal{E}_\infty$ with respect to the vague topology on the space of measures on $\R^{-\infty}\times \R^{-\infty,+\infty}$. Moreover, $\P(\mathcal{E}_n = \emptyset) \to \P(\mathcal{E}_\infty = \emptyset)$ as $n\to\infty$.
\end{theorem} 

% In the following for any simple point process $\nu$ on $\mathbb{R}^2$, let $
%  X^{*}(\nu) \coloneq  \inf \{x: (x,h) \in \nu \} \in [-\infty, \infty)$. 
% Moreover, if $ X^{*}(\nu) >-\infty$ and the infimum is attained, define  $
%   H^{*}(\nu) = \inf\left\{ h : (X^{*}(\nu),h) \in \nu \right\}$. 
% With this notation, provided $\mathcal{P}(1,n) \neq \emptyset$, we have $X^{*}(\mathcal{E}_{n}) = \mathbb{X}_{n}^{*} - \frac{1}{\alpha} \ln n $ and $H^{*}(\mathcal{E}_{n}) = \frac{ \mathbb{H}^{*}_{n } - \gamma \ln n}{\sqrt{\beta \ln n}}$. 

\begin{corollary}
Under the assumptions of Theorem~\ref{thm1}, we have the convergence
 \begin{align}
  \lim_{n \to \infty}   \Big(\mathbb{X}^{*}_{n} -\frac{1}{\alpha} \ln n, \, & \frac{ \mathbb{H}^{*}_{n } - \gamma \ln n}{\sqrt{\beta \ln n}}\Big)   \mid \mathcal{P}(1,n) \neq \emptyset \\
  &   =   (X^{*}, H^{*})  \mid   W\tilde{W}>0 \  \text{ in distribution,} 
 \end{align} 
 where $ (X^{*}, H^{*}) = \argmin\left\{ x : (x,h) \in \mathcal{E}_\infty \right\}.$ 
\end{corollary}
\begin{proof}
We remark that the function $\nu\mapsto \argmin\left\{ x : (x,h) \in \nu \right\}$ is continuous $\P_{\mathcal{E}_\infty}(\cdot\,|\,W\tilde W > 0)$-almost surely, since, on the event $W\tilde W > 0$, the process $\mathcal{E}_\infty$ has a unique atom $(x,h)$ which minimizes $x$ almost surely, and since $\nu\mapsto \nu([-\infty,x]\times \R^{-\infty,+\infty})$ is continuous $\P_{\mathcal{E}_\infty}$-almost surely, by \cite[Lemma~4.1]{KallenbergRandomMeasures2017}. The result follows from Theorem~\ref{thm1} by an application of the continuous mapping theorem and using the fact that $\P(\mathcal{P}(1,n) \neq \emptyset) = \P(\mathcal{E}_n \neq \emptyset)\to \P(\mathcal{E}_\infty \neq \emptyset) = \P(W\tilde W > 0)$.
\end{proof}

\begin{remark}\label{rmk-lattice-1} 
If the weight distribution $X$ satisfies  \eqref{def-alpha}
%-\eqref{eq-assumption-3}, 
but not  \eqref{eq-assumption-4}, i.e., it is arithmetic, then  $\mathbb{X}^{*}_n$ is supported on the lattice so that  one can not except  $\mathbb{X}^{*}_n- \frac{1}{\alpha}\ln n$ still converges weakly. 
 Instead, one can characterize its subsequential limits. For the case of non-negative weights, this was established in \cite[Theorem 1.1 (b)]{DSS23}.
 Assume that $X \in \mathbb{R}$ is arithmetic with span $d>0$. For any $ \vartheta \in [-d,0]$, let $N_{\vartheta} \subset \mathbb{N}$ be a  sequence such that 
\(
\lim_{N_{\vartheta} \ni n \to \infty} d \left\lfloor \frac{ \ln n }{d \alpha} \right\rfloor - \frac{\ln n}{\alpha} = \vartheta .
\)
Define the limiting intensity measure
\begin{equation}\label{Lambda-d}  
   \Lambda_{d}(\dif x,\dif h) =  \gamma \cdot \left( \alpha d  \sum_{j \in \mathbb{Z}} e^{\alpha j d} \delta_{j d}(\dif x) \right) \otimes  \left( \frac{1}{\sqrt{2 \pi}} e^{-\frac{h^2}{2}}   \dif h \right) . 
\end{equation}
Then under  assumption \eqref{def-alpha},
%-\eqref{eq-assumption-3}, 
% we have  
% \[
% \lim_{ N_{\vartheta} \ni n \to \infty}
% \sum_{p \in \mathcal{P}(1,n)} \delta_{(  X(p) - d\lfloor \frac{\ln n}{d \alpha} \rfloor ,\frac{ H(p)- \gamma \ln n}{\sqrt{\beta \ln n}}   )}  \mid \mathcal{P}(1,n) \ne \emptyset  \overset{\mathrm{law}}{=} \mathrm{Poisson}\Bigl(W \tilde{W} e^{\alpha \vartheta} \Lambda_{d} \Bigr) | W\tilde{W}>0 .  
% \]
along the subsequence $N_{\vartheta}$, the point process
\[
\sum_{p \in \mathcal{P}(1,n)} \delta_{\left(  X(p) - d\lfloor \frac{\ln n}{d \alpha} \rfloor ,\frac{ H(p)- \gamma \ln n}{\sqrt{\beta \ln n}}   \right)}
\] 
converges in law to the Cox process with intensity measure $W \tilde{W} e^{\alpha \vartheta} \Lambda_{d}$, and, furthermore, $\P(\mathcal{P}(1,n)\ne \emptyset) \to \P(W\tilde W \ne 0)$.
The approach in this paper can be adapted straightforwardly to obtain this result. See also Remark~\ref{rmk-lattice-2}.
% with the following modifications:
% \begin{itemize}
%   \item We should replace all the recentering term $\frac{1}{\alpha} \ln n$ by  $  d\lfloor \frac{\ln n}{d \alpha} \rfloor $
%   \item We replace all the intensity $\Lambda$ by $\Lambda_{d}$.
%   \item The statement of Lemma \ref{lem-renewal-function} need to be changed accordingly, see Remark \ref{rmk-lattice-2}.
% \end{itemize} 
% By implementing these changes and taking the limit along  subsequence $n \in N_{\vartheta}$, the left proof in this paper yields the claimed result.
\end{remark}

\subsection{Motivations and related work} 
Our motivation arises from the study of first-passage percolation (FPP) on random graphs, a classical probabilistic model used to analyze the spread of fluid flow or information through some (random) media. 
In this setting, the edges of a graph represent the possible pathways for the flow, and each edge is assigned a random passage time describing the duration required to traverse it.
The first-passage distance (or time) between two vertices is defined as the minimal total passage time needed to travel from one vertex to the other.
FPP investigates the geometric and probabilistic structures induced by these random distances. Since the pioneering work of Hammersley and Welsh \cite{HW65}, the model has been extensively studied, especially on the integer lattice $\mathbb{Z}^d$, and  we refer to the survey \cite{ADH17} for a comprehensive overview.

Beyond the lattice setting, understanding how fluids or information propagate through complex real-world networks naturally motivates the study of FPP on various random graphs.   As early as 1999, Janson  \cite{Jan99}  investigated  the minimum-weight path on complete graph with random positive weights (including exponential and uniform  edge weights.). Subsequent studies on the complete graph include \cite{BvdHH12,BvdHH13}.    
Bhamidi, van der Hofstad and Hooghiemstra further developed the theory by investigating FPP on the configuration model under both exponential and general positive edge weights with a density \cite{BvdHH10b,BvdHH10,BvdHH17b,BvdHH17}. These authors  extended their results to Erd\H{o}s--R\'{e}nyi graphs in \cite{BvdHH11} for exponential weights, with subsequent generalizations to arbitrary non-negative weights by Daly, Schulte, and Shneer \cite{DSS23}.  
Moreover  Komjáthy and Vadon  \cite{KV16} studied FPP on the Newman–Watts small world model  and  
Kistler, Schertzer and Schmidt \cite{KSS20a,KSS20b} studied oriented FPP on the hypercube in the limit of large dimensions. 

All the studies mentioned above deal with non-negative edge weights. However, in the general shortest path problem, the edge weights could be negative values, which could model  a gain instead of a loss or cost. This leads us to consider random graphs with signed edge weights.    This brings a natural difficulty: if a cycle with a negative total weight exists, one could keep going around it to make the total cost decrease without bound. To avoid this, we restrict attention to simple paths. Still, when negative edges become too abundant compared with non-negative edges, one should expect the behavior to be very different. For example, in the extreme case that the weights are equal to a negative constant, then the path with smallest weight is simply the longest path, which is known to be of length proportional to $n$ \cite{AKS81}.  Our assumption \eqref{def-alpha} guarantees that locally around a point, the weight of all paths originating from that point are ultimately positive and in fact grow linearly in their length, see Section~\ref{sec:local_weak_convergence} for details. Under this assumption, we show that the minimum-weight path statistics still belongs to the same universality class as in  \cite{BvdHH17} and \cite{DSS23}. 

We now provide a glimpse into our method and compare it to previous methods to study FPP on sparse graphs. A classical approach, followed for example by Bhamidi, van der Hofstad and Hooghiemstra  \cite{BvdHH10b,BvdHH10,BvdHH17b,BvdHH17} or by Komjáthy and Vadon  \cite{KV16} is to describe the growth of the balls of radius $t\ge 0$ for the first passage percolation distance around a given point by a coupling to a continuous-time branching process. The shortest distance between two points is then given by (twice) the first time the two growing balls around the two points intersect. In order to use this idea effectively, it is convenient if the weights are exponentially distributed. In the general case, this method is more complicated to use, since the continuous-time branching process used in the coupling is not Markovian anymore. An additional difficulty is the need to couple to the branching process for a long time, at which cycles start to appear and which need to be taken care of. 

In order to deal with completely general (non-negative) weights, Daly, Schulte, and Shneer \cite{DSS23} used a radically different method, based on a method of moments. They were able to calculate moments of quantities counting the number of extremal paths and their hopcount and showed that these moments satisfy a certain recurrence. They were then able to identify this recurrence in terms of the Cox process $\mathcal E_\infty$, based on a recurrence of the moments of $W$. It is worth to point out that their moment calculations work \emph{vanilla}, i.e.~without any truncation. One reason for this is that, in the case of non-negative weights, the Biggins martingale limit $W$ of the branching random walk has exponential tails \cite{RoslerFixedPoint1992}. This is not true in the case of signed weights, in fact, under common assumptions, $W$ has a polynomial tail \cite{Liu2000}. One could still hope to adapt the method of moments by truncating paths which become negative, or smaller than some constant $-c$. Indeed, if one does this for the branching random walk, one recovers exponential tails for $W$, as was recently shown by the authors \cite{MaMaillardExponentialMoments2025}. However, it seems not clear how to establish a recurrence for the moments, because of the need to keep track of the starting point of the branching random walk.

In this paper, we therefore propose a method which is closer in spirit to classical methods, but with several twists. First, we explore clusters locally around points, but not as balls with respect to the FPP distance, but rather with respect to graph distance, describing them by a branching random walk instead of a continuous-time branching process. This in particular allows to easily handle signed weights. Second, we do not insist on coupling the clusters to the branching random walk until they meet, but only up to a certain \emph{graph distance} growing slowly with $n$. This allows us not to be bothered with cycles at this stage, which otherwise is a major source of difficulty. Third, we patch the two clusters together by intermediate paths. These paths will be almost independent from each other, so that they will account for the Poisson process. To prove this, we use the Chen-Stein method, which reduces the problem to the calculation of some (truncated) first and second moments. 

Our approach is partly inspired by Kistler, Schertzer and Schmidt \cite{KSS20a,KSS20b} on oriented FPP on the hypercube (with exponentially distributed weights). This is a dense setting and the underlying graph is deterministic, whereas we work on a random, sparse graph. While the results and scalings are different, the underlying idea of conditioning on neighborhoods around the two vertices and using the Chen-Stein method to deal with the paths connecting the two neighborhoods still works in our setting. One difference is that we do not apply the Chen-Stein method as is, since we restrict to a certain good set for the second moment calculations. We therefore derive a variant of the method that incorporates a certain global dependency and which is possibly of independent interest.

To conclude, \textit{we believe that the method in this article paves the way to effectively studying first passage percolation with general, possibly signed weights on other sparse, locally tree-like  graphs, such as the configuration model.}  We also believe that they allow to deal with dense settings, if the FPP distances can be locally approximated by a branching random walk (with infinite number of offspring). Indeed, the method is very robust, relying mainly on (truncated) first- and second-moment computations. It should also extend to the study of more complicated quantities, such as the vector of FPP distances between three or more points.
  
\subsection{Roadmap of the proof}
\label{sec-prood-road-map}
The starting point of our proof is the following lemma, which indicates that to establish weak convergence of a sequence of point processes, it suffices to verify the convergence of the number of points contained in each rectangle.

\begin{lemma}
\label{lem:kallenberg}
Consider the set of rectangles of the form
\[
\mathcal R = \{(x,x']\times (h,h']: x,x'\in\R^{-\infty},\ x\le x',\ h,h'\in\R^{-\infty,+\infty},\ h\le h'\}.
\]
Then $\mathcal{E}_n$ converges in law to $\mathcal{E}_\infty$ with respect to the vague topology on measures on $\R^{-\infty}\times \R^{-\infty,+\infty}$ if and only if $\mathcal{E}_n(R)$ converges in law to $\mathcal{E}_\infty(R)$ for every $R\in\mathcal{R}$.
\end{lemma}
\begin{proof}
We aim to apply Theorem~4.15 in Kallenberg~\cite{KallenbergRandomMeasures2017}. Define the class
\begin{align*}
\tilde{\mathcal R} &= \Big(\{(x,x']: x,x'\in\R^{-\infty},\ -\infty < x\le x'\}\cup\{[-\infty,x]: x\in \R^{-\infty}\}\Big)\\
&\times \Big(\{(h,h']: h,h'\in\R^{-\infty,+\infty},\ -\infty < h\le h'\}\cup\{[-\infty,h]: h\in \R^{-\infty,\infty}\}\Big)
\end{align*}
Then $\tilde{\mathcal R}$ is, in the language of that reference, a dissecting semi-ring of the Borel sets on $\R^{-\infty}\times \R^{-\infty,+\infty}$ and we have for every $R\in\tilde{\mathcal{R}}$, $\E[\mathcal{E}_\infty(\partial R)] = \E[W\tilde W\Lambda(\partial R)] = 0$. Theorem~4.15 in Kallenberg~\cite{KallenbergRandomMeasures2017} (or the remark that precedes it) then implies the statement with $\tilde{\mathcal R}$ in place of $\mathcal R$. Now notice that $\mathcal{E}_n$ and $\mathcal{E}_\infty$ do not contain points $(x,h)$ with $x = -\infty$ or $h = -\infty$, so that the statement also holds with $\mathcal{R}$ instead.
\end{proof}

Notice that the law of $\mathcal{E}_{\infty}(R)$   is a Poisson distribution, but with a \textit{random} intensity $W\tilde{W} \Lambda(R)$. Inspired by the previous study \cite{BvdHH11} and \cite{KSS20b},   the randomness in the intensity, namely  $W$ and $\tilde{W}$,   comes from the random evolution within the local neighborhood of the  
vertices $1$, $n$.  
Thus,  to  establish the weak convergence of $\mathcal{E}_{n}(R)$,  we therefore   first reveal or explore the local neighborhood of $\{1,n\}$. Then conditionally on this local information, the limiting distribution should be  a Poisson random variable.  There exist various tools for proving Poisson convergence; here the Chen–Stein method seems to be simple and elegant.

We now give a precise description of the framework, starting with the relevant notation. 
For any vertex subset $A \subset [n]$ and  a radius $\rho \in \mathbb{N}$, let $\mathcal{S}_{\rho}(A)$ denote the set of vertices at a distance of $\rho$ from $A$. The $\rho$-neighborhood of $A$ in $G_n$, denoted by $\mathcal{S}_{\leq \rho}(A)$, is the (weighted) subgraph induced by the union of the $k$-spheres $\mathcal{S}_k(A)$ for $k$ from $0$ to $\rho$.
Throughout this paper, we let $(r_n)_{n\ge1}$ be a sequence of natural numbers such that\footnote{We could in fact allow $r_n \le \delta \log n$ for some small $\delta>0$, at the expense of some more complicated arguments. In particular, we would need to extend Lemma~\ref{lem-renewal-function} to a wider range of $x$.}
\begin{equation}
\label{eq:r_n}
r_n\to\infty \quad\text{and}\quad r_n = o(\sqrt{\log n}).
%  r_{n} \coloneq  1+ \left\lceil \, (\ln_{+} (\ln n) )^{1/2}\, \right\rceil \text{ for } n \ge 1. 
 \end{equation}
 %\pascal{The choice of $r_{n}$ is made by our preference. Indeed any $1 \ll r_n \le \delta \log n$ for some $\delta$ small should be enough.} \heng{Seems we need $r_{n}= o(\sqrt{\log n})$. }
 Moreover, we  introduce a good event $\mathsf{G}^{1}_n $ for the $r_{n}$-neighborhoods of vertices $1$ and $n$:
\begin{align}
    \mathsf{G}^{1}_n  &\coloneq   \big\{  \mathcal{S}_{\le r_{n}}(1) \text{ and } \mathcal{S}_{\le r_{n}}(n)  \text{ are both trees, and } |\mathcal{S}_{\le r_{n}}(1)|  \vee  |\mathcal{S}_{\le r_{n}}(n)|  \le (2 \lambda)^{r_{n}} \big\} \\
   & \cap \big\{ \mathcal{S}_{\le r_{n}}(1, n)  \text{ is a forest of two trees}  \big\}. 
\end{align}
For technical reasons, we'll introduce two additional good events $\mathsf{G}^{2}_{n},\mathsf{G}^{3}_{n}$ in \S \ref{sec-good}; and  our analysis will always assume the occurrence of   $\mathsf{G}_{n}\coloneq  \mathsf{G}^{1}_n \cap \mathsf{G}^{2}_{n} \cap \mathsf{G}^{3}_{n}$. Lemma  \ref{lem-good-event} shows that the good event $\mathsf{G}_{n}$ happens with high probability. It makes use of the fact that the neighborhood of vertex $1$ in $\mathcal{G}_{n}$ converges  in the local weak sense  to a Galton-Watson   tree with $\mathrm{Poisson}(\lambda)$ offspring distribution. In the following, let 
\begin{equation*}
  \pi_\lambda(k) = \frac{\lambda^k}{k!}e^{-\lambda} \text{ for all } \lambda >0 \text{ and } k \in \mathbb{N}.
\end{equation*} 
Notice that for $\lambda,\mu>0 $, 
\begin{equation}\label{eq-tvd-poi}
   \sup_{k\ge 0}\bigl|\pi_\lambda(k)-\pi_\mu(k)\bigr|
   \le \P( \mathrm{Poi}(|\lambda-\mu|) \ge 1  ) \le  |\lambda-\mu|.  
\end{equation} 

 Given $R \in \mathcal{R}$, we bound the difference of the probability distribution function of $\mathcal{E}_{n}(R)$ with $\mathcal{E}_{\infty}(R)$ as follows:
\begin{align}
  & \left| \P\left(  \mathcal{E}_{n}(R) = k   \right) -  \P\left(  \mathcal{E}_{\infty}(R) = k \right) \right|     \\
  & \le  \P(\mathsf{G}_{n}^{c} ) + \E \left[ \left| \,  \P\left[  \mathcal{E}_{n}(R) = k ,  \mathsf{G}_{n} \mid \mathcal{S}_{\le r_{n}}(1,n) \right] - \pi_{\lambda_{n}(R) }(k)  \ind{  \mathsf{G}_{n}} \right|  \right]  \label{eq-avoid-prob-error-1}\\ 
  &  \qquad   +   \left|\E \left[      \pi_{\lambda_{n}(R) }(k)   \ind{  \mathsf{G}_{n}}  \right] -  \E \left[  \pi_{ W \tilde{W}\Lambda(R) }(k)   \right] \right|  \label{eq-avoid-prob-error-2} 
\end{align}
where 
\begin{equation}
 \lambda_{n}(R) \coloneq  \E \left[  \mathcal{E}_{n}(R) \mid \mathcal{S}_{\le r_{n}}(1,n) \right] \ge 0 .
\end{equation} 
Therefore, by Lemma \ref{lem:kallenberg}, to establish the weak convergence $\mathcal{E}_{n} \to \mathcal{E}_{\infty}$, 
it suffices to show that  the error terms in \eqref{eq-avoid-prob-error-1} and \eqref{eq-avoid-prob-error-2} vanish as $n \to \infty$.

\medskip

We first explain the idea for proving that  \eqref{eq-avoid-prob-error-2} vanishes.
Notice that when $\mathsf{G}^{1}_n$ occurs, for any vertex $u \in \mathcal{S}_{r_{n}}(1)$, the there exists a unique path from $1$ to $u$ within the subgraph $\mathcal{S}_{\le r_{n}}(1,n)$. We refer to this path as $[1,u]_{r_{n}}$ or just $[1,u]$ for simplicity if there's no risk of confusion.
 Similarly, for any  $v \in \mathcal{S}_{r_{n}}(n)$, $[v,n] $  denotes the unique path from $v$ to $n$ in  $\mathcal{S}_{\le r_{n}}(1,n)$. Hence the following two quantities are well-defined:
\begin{equation}
  W_{r_{n}}  \coloneq  \ind{\mathsf{G}^{1}_n}\sum_{u \in  
  \mathcal{S}_{r_{n}}(1)} e^{-\alpha X( [1, u])} \text{ and } \tilde{W}_{r_{n}} \coloneq  \ind{\mathsf{G}^{1}_n} \sum_{v \in  
  \mathcal{S}_{r_{n}}(n)} e^{-\alpha X([n,v])} . 
\end{equation}  
A simple computation yields that the (conditional) expectation of $\mathcal{E}_{n}(R)$ is related to the (truncated) exponential renewal function of the random walk with step distribution $X$ (see Lemma \ref{lem-set-A}). By using the renewal theory for random walk with positive jumps (see Lemma~\ref{lem-renewal-function}), in Lemma \ref{lem-cond-set}  we will show  that 
 \begin{equation}
  |\lambda_{n}(R) - W_{r_{n}} \tilde{W}_{r_{n}} \Lambda(R) | \to 0 \text{ in probability as } n \to \infty .
 \end{equation}
Then \eqref{eq-tvd-poi} and bounded convergence theorem imply that 
\[  |\E [      \pi_{\lambda_{n}(R) }(k)   \ind{  \mathsf{G}_{n}}  ]- \E [      \pi_{ W_{r_{n}} \tilde{W}_{r_{n}} \Lambda(R) }(k)] | \to 0 \]
 Since the local weak limit of the weighted graph $\mathcal{G}(n,\lambda/n)$ is a branching random walk,   Corollary \ref{lem-W-r_{n}-W} shows that
\begin{equation}
  ( W_{r_{n}} , \tilde{W}_{r_{n}}) \text{ converges weakly to }  (W,\tilde{W})  \text{ as } n \to \infty.
\end{equation}
In particular $ |  \E [      \pi_{ W_{r_{n}} \tilde{W}_{r_{n}} \Lambda(R) }(k) ] - \E [   \pi_{ W \tilde{W}\Lambda(R) }(k) ] | \to 0$, which establishes \eqref{eq-avoid-prob-error-2}.

The proof of \eqref{eq-avoid-prob-error-1} is based on the Chen–Stein method for Poisson approximation, originally developed by Chen \cite{Chen75} and subsequently refined and simplified by Arratia, Goldstein, and Gordon \cite{AGG89}.  This work provides a robust framework for approximating sums of dependent indicator random variables with a Poisson distribution, with broad applications across various models (see, e.g., \cite{CCH15,CCH16,KSS20a,KSS20b,ST20,FTWY24}; the list is by no means exhaustive).
Compared with the   formulations  in \cite[Theorem4.1]{Chen75}, \cite[Theorem 1]{AGG89} or \cite[Theorem 4.7]{Ross11}, in the following version, we have an additional   indicator variable $\chi$ in the covariance terms, see \eqref{eq-Stein} (and we condition on a $\sigma$-algebra $\mathcal F$, but this is a trivial generalization). For completeness we provide a proof in Appendix \ref{app1}.

\begin{lemma}[Chen--Stein method for Poisson Approximation]\label{Stein_method}
Let  \((X_i)_{i\in I}\) be a family of Bernoulli random variables with finite index set $I$. 
Let \( \mathcal F\) be a  sub-$\sigma$-algebra.  For each \(i\in I\), suppose that a conditionally dissociating neighborhood \(N_i\subset I\) is given, that is, conditionally on \(\mathcal F\),  $X_{i}$ is independent of  $\{X_j : j\notin \bar{N}_{i}\}$ where $\bar{N}_{i}\coloneq N_{i} \cup \{i\}$.
Define
\[
Y \coloneq  \sum_{i\in I} X_i, \text{ and }
\lambda \coloneq  \sum_{i\in I} \E_{\mathcal{F}} [X_{i}]= \sum_{i\in I} \P_{\mathcal{F}} \left( X_{i}=1 \right) .
\]
Let $\chi$ be another Bernoulli random variable (not necessarily $\mathcal{F}$-measurable). 
Then  
\begin{equation}\label{eq-Stein}
  \begin{aligned}
   &  \mathrm{d}_{\mathrm{TV}, \mathcal F}(\chi Y, \mathrm{Poi}(\lambda) ) \coloneq  \sup_{A \subset \mathbb{N}}
  \bigg|\, \P_{\mathcal{F}}(\chi Y \in A)-  \sum_{k \in A} \pi_\lambda(k)   \,\bigg| \\
  & \le  2 \P_{\mathcal{F}}(\chi=0)  +  \sum_{i\in I}\sum_{j\in \bar{N}_i}
  \E_{\mathcal{F}} [X_i] \, \E_{\mathcal{F}} [X_j ]
  + \sum_{i\in I}\sum_{j\in N_i}
  \E_{\mathcal{F}} [   X_iX_j \chi ] .
  \end{aligned} 
\end{equation}
\end{lemma}
 
To apply this Lemma, we rewrite $\mathcal{E}_{n}(R) $ as summation of Bernouli random variables:
\begin{equation*}
  \mathcal{E}_{n}(R) = \sum_{p} I(p;R) 
\end{equation*} 
where  $p$ represents the potential paths starting from $1$ to $n$ (see \eqref{eq-expression-EnR} for the precise definition) and $I(p,R)$ indicates that the potential path $p$ contributes to $\mathcal{E}_{n}(R)$, i.e.,  $X(p)$ and $H(p)$, after rescaling, falls in set $R$. 
We will apply Lemma \ref{Stein_method} with 
\begin{equation*}
  Y=\mathcal{E}_{n}(R) \,, \ \chi= \ind{\mathsf{G}_{n}} \,, \ \mathcal{F}=\mathcal{S}_{\le r_{n}}(1,n) ;
\end{equation*}
 and the dissociating neighborhood of for a potential path $p$, denoted by  $\mathcal{N}_{p}$,  consists of other potential paths that share at least one edge with $p$ outside $\mathcal{S}_{\le r_{n}}(1,n)$. 

 In Lemma \ref{lem-correlation-1}  we will show that  $\sum_{p} \sum_{p' \in \mathcal{N}_{p}} \mathbb{E}_{\mathcal{F}}[ I(p;R)]  \mathbb{E}_{\mathcal{F}}[I(p';R)] $  vanishes as $n \to \infty$.
 This first-moment estimate follows from a straightforward counting argument on $|\mathcal{N}{p}| $ together with a renewal-theoretic result (see Lemma~\ref{lem-renewal-function}).   In Lemma \ref{lem-correlation-2} we further prove that 
 the term $\sum_{p} \sum_{p' \in \mathcal{N}_{p}}  \mathbb{E}_{\mathcal{F}}[ I(p;R)I(p';R)] $ also vanishes, based on a more refined counting of  subclasses of $p'\in \mathcal{N}_{p}$ with specific correlation structures between $I(p';R)$ and $I(p;R)$. We will see that the introduction of the good event $\mathsf{G}_{n}$ within the correlation terms is necessary when $\psi_{X}(2 \alpha) \ge \psi_{X}(\alpha)$. 
In contrast, when $\psi_{X}(2\alpha) < \psi_{X}(\alpha)$, which includes the case $X \ge 0$, we may simply take $\chi = 1$.

%\pascal{add more details, notably about the first moment computations mention Lemma~\ref{lem-renewal-function} and the second moment computations.}

% \subsection{Overview of paper} 
% \heng{To be added.}

\section{Asymptotic of Expectations}
 
\subsection{Preliminary: Renewal function.}
Let $(S_{n})_{n \ge 0}$ denote a random walk with step distribution $X$ and initial position $S_0=0$. We define the renewal function 
\begin{equation}
  V(x)\coloneq   \sum_{k \ge 1} \lambda ^{k} \P( S_{k} \le x ) \quad  \text{ for all } x \in \mathbb{R}. 
\end{equation}
Furthermore, for $h\in\R^{-\infty,+\infty}$, let 
\begin{equation}
    \label{eq:def_knh}
k_{n}(h)\coloneq  \lfloor  \gamma \ln n + h \sqrt{\beta \ln n} \rfloor \vee 0.
\end{equation}

\begin{lemma}\label{lem-renewal-function} 
  \begin{enumerate}[(i)]
 \item There exists $C<\infty$ such that for all $x \ge 0$, 
\begin{equation}\label{eq-renewal-function-bnd}
 V(x) \le C e^{\alpha x}  
\end{equation} 
Moreover, we have the following limits:
% \[
% \lim_{x \to \infty} \frac{V(x)}{e^{\alpha x}} = \gamma \quad\text{and}\quad \lim_{x \to -\infty} \frac{V(x)}{e^{\alpha x}} = 0.
% \]
\[
\lim_{x \to \infty} \frac{V(x)}{e^{\alpha x}} = \gamma,
\]
and there exists $\alpha' > \alpha$ such that
\[
\lim_{x \to -\infty} \frac{V(x)}{e^{\alpha' x}} = 0.
\]
 \item Let $x\in\R$ and $h \in \R\cup\{+\infty\}$. Let $(\Gamma_{n})_{n \ge 1}$ be any sequence of positive numbers such that $\Gamma_{n}=o(\sqrt{\log n})$.  Then, as $n\to\infty$,   
\begin{equation}\label{lem-renewal-function-lim}
  \lim_{n \to \infty} \sup_{|x| \le \Gamma_{n}}  \left|\frac{1}{n \, \Lambda((-\infty,x]\times(-\infty,h])} \sum_{k=1}^{k_{n}(h)}  \lambda^k
 \P \left( S_k \le \frac{1}{\alpha} \ln n + x \right)  -1\right| = 0 .
\end{equation}
% Moreover, for any $x_0\in\R$, the convergence holds uniformly over all $x\le x_0$.
  \end{enumerate}
\end{lemma}

\begin{proof}[Proof of Lemma \ref{lem-renewal-function}]
In the case $X \ge 0$ the lemma is proven in \cite[Section 2.2]{DSS23}, based on \cite[Theorem 2.2]{IM10} which essentially proves assertion (i) and covers the general case. For completeness, we provide a detailed proof. 

  Let us denote $( \hat{S}_{k})$ the random walk with initial position $ \hat{S}_{0}=0$ and step distribution $\hat{X}$  given by 
\begin{equation}\label{eq-hat-S}
  \E[  f( \hat{X}) ] = \lambda \E[ e^{-\alpha X} f(X) ]  \ , \  \forall \ f \in \mathrm{C}_{b}(\mathbb{R}). 
\end{equation} 
In particular,   $\E[  f( \hat{S}_{k}) ] = \lambda^{k}\E[ e^{-\alpha S_{k}} f({S}_{k}) ] $ for any $k \ge 1$ and 
\begin{equation}
\label{eq:exp_hatS1}
\E[\hat{S}_1] = \lambda \E[ \hat{S}_1 e^{-\alpha S_{1}}] = -\psi_{X}'(\alpha) > 0.
\end{equation}
using \eqref{def-alpha} for the last equality.

We now have
\begin{equation}
    V(x) = \sum_{k \ge 0} \E\left[  e^{\alpha \hat{S}_{k} }  \ind{\hat{S}_{k} \le  x } \right] = e^{\alpha x} \sum_{k \ge 0} \E\left[  e^{\alpha( \hat{S}_{k}-x) }  \ind{\hat{S}_{k} \le  x } \right].
\end{equation}
The first two statements of the first part now follow from the two-sided renewal theorem, together with the fact that the function $y\mapsto e^{-\alpha y}$ is directly Riemann integrable, see Feller~\cite[Sections XI.1 and XI.9]{Feller1971}. For the third statement of the first part, by assumption~\eqref{def-alpha}, there exists $\alpha'>\alpha$ such that $\psi_X(\alpha') < \psi_X(\alpha)$ and $\psi_X'(\alpha') <0$. Set  $\rho = e^{\psi_X(\alpha') - \psi_X(\alpha)} < 1$. Denote by $\tilde{S}_k$ the random walk with step distribution $\tilde{X}$ where 
\[
\E[f(\tilde X)] = e^{-\psi_X(\alpha')}\E[e^{-\alpha' X} f(X)] = \frac1\rho \E[e^{-(\alpha'-\alpha)\hat X}f(\hat X)].
\]
We then have
\[
V(x) =  \sum_{k \ge 0} \rho^{k}\E\left[  e^{\alpha' \tilde{S}_{k} }  \ind{\tilde{S}_{k} \le  x } \right] \le \sum_{k \ge 0}\E\left[  e^{\alpha' \tilde{S}_{k}} \ind{\tilde{S}_{k} \le  x }\right].
\]
Since $\E[\tilde X] = -\psi_X'(\alpha') > 0$, the two-sided renewal theorem implies as above that
\[
V(x)e^{-\alpha' x} \to 0,\quad x\to\infty,
\]
which finishes the proof of point (i).

To prove assertion (ii), note that for $h = +\infty$, it follows from (i). Hence, by additivity, it is enough to show that for every $h\in\R$,
\begin{equation}\label{eq:renewal-function-lim-toshow}
  \lim_{n \to \infty} \sup_{|x| \le \Gamma_{n}}  \left|\frac{1}{n \, \Lambda((-\infty,x]\times(h,\infty))} \sum_{k=k_{n}(h)+1}^\infty  \lambda^k
 \P \left( S_k \le \frac{1}{\alpha} \ln n + x \right)  -1\right| = 0 .
\end{equation}
We have,
\begin{equation}
\sum_{k = k_n(h)+1 }^{\infty} \lambda^k
 \P \Big( S_k \le \frac{1}{\alpha} \ln n + x \mid S_{k_{n}(h)}  \Big) 
  = \lambda^{k_{n}(h)}    V\Big(  \frac{1}{\alpha} \ln n + x - S_{k_{n}(h)} \Big) .
\end{equation}
 Thus by using the change of measure,  we obtain 
 \begin{align}
\sum_{k = k_n(h)+1 }^{\infty} \lambda^k
 \P \Big( S_k \le \frac{1}{\alpha} \ln n + x  \Big) &= \lambda^{k_{n}(h)} \E \bigg[    V\Big(  \frac{1}{\alpha} \ln n + x - S_{k_{n}(h)} \Big)    \bigg] \\
 &=   n e^{\alpha x}  \,  \E \bigg[  \frac{  V(  \frac{1}{\alpha} \ln n + x - \hat{S}_{k_{n}(h)} )  }{e^{\alpha [  \frac{1}{\alpha} \ln n + x - \hat{S}_{k_{n}(h)} ]}} \bigg] . 
 \end{align}
We set $\xi_n = (\hat{S}_{k_{n}(h)} -  \E[ \hat{S}_{1}] k_{n}(h))/\sqrt{\Var(\hat{S}_1) k_n(h) },$
so that
 \begin{equation}
   \hat{S}_{k_{n}(h)} = \E[ \hat{S}_{1}] k_{n}(h) + \xi_{n}\sqrt{\Var(\hat{S}_1) k_n(h) }
 \end{equation} 
According to the central limit theorem, $\xi_n$ converges in distribution to a standard normal random variable.
Recall that $\E[ \hat{S}_{1}]= |\psi_{X}'(\alpha)|$ and $\Var( \hat{S}_{1} )=\psi_{X}''(\alpha)$, and 
$k_{n}(h)$ is the integer part of $ \frac{1}{\alpha |\psi'_{X}(\alpha)|} \ln n + h \big[\frac{\psi''_{X}(\alpha)}{\alpha |\psi'_{X}(\alpha)|^3}  \ln n \big]^{1/2}$. Writing $c_{\alpha} = \big[ \frac{\psi''_{X}(\alpha)}{\alpha |\psi'_{X}(\alpha)|} \big]^{1/2}$,  simple calculations yield 
\begin{align}
   \frac{1}{\alpha} \ln n + x - \hat{S}_{k_{n}(h) } &= \frac{1}{\alpha} \ln n + x - \Big[ \, \frac{1}{\alpha} \ln n  +  c_{\alpha} (\xi_{n}+h+o_{n}(1)) \sqrt{\ln n  } \, \Big]\\
   & = x - c_{\alpha}  (\xi_{n}+h+o_{n}(1)) \sqrt{\ln n } .
\end{align}
Using \eqref{eq-renewal-function-bnd} and the asymptotics of $V(x)$ for $x\to \pm \infty$ from point (i), we can apply the dominated convergence theorem to get
\begin{equation}
 \lim_{n \to \infty} \E \left[  \frac{  V(  \frac{1}{\alpha} \ln n + x - \hat{S}_{k_{n}(h)} )  }{e^{\alpha [  \frac{1}{\alpha} \ln n + x - \hat{S}_{k_{n}(h)} ]}} \right] =  \lim_{n \to \infty} \P( \xi_{n}+h < 0) \lim_{y \to \infty} \frac{V(y)}{e^{\alpha y}} = (1-\Phi(h))\gamma, 
\end{equation}
uniformly in $|x| \le \Gamma_{n}=o(\sqrt{\log n})$,
where $\Phi(\cdot)$ denotes the c.d.f. of a standard Gaussian distribution.
%This proves the convergence statement for fixed $x$. The uniformity in $x\le x_0$ is a direct consequence, since the functions are monotone in $x$ and non-negative and the right-hand side is continuous (this uses the so-called ``Dini's second theorem'').
This completes the proof of \eqref{eq:renewal-function-lim-toshow} and therefore finishes the proof of the lemma.
\end{proof}

\begin{remark}\label{rmk-lattice-2} 
As discussed in Remark \ref{rmk-lattice-1}, when the weight distribution $X$ has a span $d>0$, we should apply the key renewal theorem in arithmetic case. Then  assertion (i) becomes $\lim_{k \to \infty} V(kd)/e^{\alpha kd} = \frac{ \alpha d}{1-e^{-\alpha d}}  \gamma$. Assertion (ii) becomes 
  \begin{equation} 
  \lim_{n \to \infty} \sup_{ j \in \mathbb{Z}, |j| \le \Gamma_{n} }\left| \frac{1}{e^{d \left\lfloor \frac{\ln n}{d \alpha} \right\rfloor  } \Lambda_{d}(R_{j d,h})} \sum_{k=0}^{k_{n}(h)}  \lambda^k
 \P \left( S_k \le d \lfloor \frac{\ln n}{d \alpha} \rfloor + d j\right)  - 1  \right| =0.
\end{equation} 
where  $\Lambda_{d}(R_{j d,h})=  \frac{\gamma\alpha d}{1-e^{-\alpha d}} e^{\alpha j d} \Phi(h) $ by \eqref{Lambda-d}.
\end{remark}

\subsection{First moment asymptotic}
In this subsection we study the asymptotic behavior of $\E[ \mathcal{E}_{n} (R  ) ]$ by using the results on renewal function obtained above.

\begin{lemma}\label{lem-set-A}
  Let $R\in\mathcal R$. Then  
 \begin{equation}
    \lim_{n \to \infty}  \E[ \mathcal{E}_{n} (R ) ] = \Lambda ( R  ) = \gamma e^{\alpha x} \,\Phi(h)   .
 \end{equation}
\end{lemma}

We begin with some necessary notation.
Let us denote by $\mathcal{P}_{\mathrm{f}}(1,n)$ all the  free paths (or potential paths) from $1$ to $n$:
\begin{equation}
  \mathcal{P}_{\mathrm{f}}(1,n) \coloneq  \left\{ p=(v_{i})_{i=0}^{\ell} : v_0=1, v_{\ell}=n; v_{i } \neq v_{j} \text{ for } i \neq j  \right\}. 
\end{equation}   
 For $p \in \mathcal{P}_{\mathrm{f}}$ and $R\in\mathcal R$,  set
\begin{equation}
  I(p;R) \coloneqq \boldsymbol 1 \left\{  p \in \mathcal{P}(1,n) \text{ and }   \left(  X(p) - \frac{1}{\alpha} \ln n , \frac{H(p)- \gamma \ln n }{\sqrt{\beta \ln n}}  \right) \in R \right\}. 
\end{equation} 
Then we can rewrite $  \mathcal{E}_{n}(R)$ as the sum of a family of dependent Bernoulli variables:
\begin{equation}
  \mathcal{E}_{n}(R) =\sum_{p \in \mathcal{P}_{\mathrm{f}}(1,n)} I(p;R) . 
\end{equation}

\begin{proof}[Proof of Lemma \ref{lem-set-A}]
  %[Proof of Lemma \ref{lem-set-A} Lemma \ref{lem-renewal-function}
By linearity of expectation, it is enough to consider $R$ of the form
\[
R = R_{x,h} = (-\infty,x] \times (-\infty,h],
\]
for $x\in\R$, $h\in \R \cup \{+\infty\}$. We have,
 \begin{equation}
\E[\mathcal{E}_{n} (R_{x,h}) ] 
    = \sum_{p \in \mathcal{P}_{\mathrm{f}}(1,n) } \E \left[   I(p;R_{x,h})  \right]  = \sum_{\ell=1}^{ k_{n} (h) } \sum_{p \in \mathcal{P}_{\mathrm{f}}(1,n) , H(p)=\ell } \E \left[   I(p;R_{x,h})  \right]  .
 \end{equation} 
Observe that, for any $p \in \mathcal{P}_{\mathrm{f}}(1,n)$ with $ H(p)=\ell$, the expectation equals
 \begin{equation}
\E \left[   I(p;R_{x,h})  \right] =  \left(\frac{\lambda}{n}\right)^{\ell} \,  \P( S_{\ell} \le \alpha \ln n + x ) .
 \end{equation}
The number of free paths of hopcount $\ell$ is given by 
\begin{equation}
|\{ p \in \mathcal{P}_{\mathrm{f}}(1,n) , H(p)=\ell  \} | = \prod_{j=0}^{\ell-1} (n-2-j)  =(n-2)_{\ell},
\end{equation}
using the Pochhammer symbol.
Combining the previous formulas with the fact that   $\frac{(n-2)_{\ell} }{n^{\ell}} \to 1$  uniformly in $1 \le \ell \le \ln^{2} n$ as $n \to \infty$, it follows that 
 \begin{align}
  \E[\mathcal{E}_{n} (R_{x,h}) ]   
    = [1+o_{n}(1)] \sum_{\ell=1}^{k_{n}(h) }\lambda^{\ell} \P( S_{\ell} \le \alpha \ln n + x ) = [1+o_{n}(1)] \Lambda ( R_{x,h}  ).
  \end{align}
  Above,   we have used assertion (ii) of  Lemma \ref{lem-renewal-function}. We now complete the proof.
\end{proof}

\subsection{Local weak convergence to branching random walk} 
\label{sec:local_weak_convergence}
A key result for our analysis is the local weak convergence of the  sparse Erd\H{o}s--R\'{e}nyi graph $\mathcal{G}(n,\lambda/n)$ to a $\mathrm{Poi}(\lambda)$ Galton–Watson tree, referred to as $\mathrm{PGW}(\lambda)$ tree. 
This has been a quite  standard and widely-used technique. In the following we follow  \cite[Lemma 2.2]{RW10}.

Given fixed  vertices $u,v \in \mathcal{G}_{n}$, we explore  
 their local neighborhoods up to a radius $r_{n}$ using a two-stage breadth-first search (BFS). First, we run a BFS starting from the root $u$.  Starting with $\mathcal{S}_0(u) = \{u\}$ at level $\rho=0$, to find $\mathcal{S}_{\rho+1}(u)$, we iterate through the vertices $w \in \mathcal{S}_{\rho}(u) \subset [n]$ one by one in increasing order of their labels. For each $w$, we examine  all edges from $w$ to vertices not yet reached.
If an edge $(v,w)$ connects to an unvisited vertex $w$, we add $w$ to $\mathcal{S}_{\rho+1}(u)$ and call this edge 'revealed.'  
This process continues until we have identified all vertices in $\mathcal{S}_{\le r_{n}}(u)$.  
Then we perform the same BFS procedure starting from $v$,   but   only exploring vertices that were not discovered in the first search.  This second stage   identifies the vertex set of $\mathcal{S}_{\le r_{n}}(v) \setminus \mathcal{S}_{\le r_{n}}(u) $.
 The two sets of revealed edges from these searches constitute the graphs $\mathcal{S}^{\mathtt{BFS}}_{\leq r_{n}}(u), \mathcal{S}^{\mathtt{BFS}}_{\leq r_{n}}(v)$. By construction, these are both  trees. We treat them as rooted, ordered trees with respective roots $u$ and $v$, where the children of any node are ordered by their increasing labels.

\begin{lemma}\label{lem-local-weak-convergence}
Let $\mathcal{T}$ be a $\mathrm{PGW}(\lambda)$-tree up to height $r_{n}$. Let $\Gamma_{n}$ be any sequence of positive numbers such that $\Gamma_n = o(\sqrt{n})$. Then for any rooted ordered trees $\tau, \tilde{\tau}$ with $|\tau|, |\tilde{\tau}| \le  \Gamma_{n}$,  
\begin{equation}
  \P \left( \mathcal{S}_{\le r_{n}}(1)  = \tau ,  \mathcal{S}_{\le r_{n}}(n)= \tilde{\tau} ,   \mathcal{S}_{\le r_{n}+1}(1) \cap \mathcal{S}_{\le r_{n}}(n) =\emptyset \right) = (1 + o(1)) \P \left( \mathcal{T} = \tau \right)  \P \left( \mathcal{T} = \tilde{\tau} \right) 
\end{equation}
 as $n \to \infty$, where $o(1)$ depends only on $\Gamma_{n}$.

\end{lemma} 

\begin{proof} 
Suppose before examining $u$ we have reached $k=k_{u}$ vertices in total so far. Then the probability of revealing $\ell=\ell_{u}$ edges in the next step is 
$\binom{n-k}{\ell} \left(\frac{\lambda}{n}\right)^{a} (1-\frac{\lambda}{n})^{n-k-\ell}$. In the PGW$(\lambda)$ tree a vertex $u$ produces $\ell=\ell_{u}$ children has probability $\frac{\lambda^{\ell}}{\ell !} e^{-\lambda}$.
 The ratio of these probabilities is 
\begin{equation}
 \frac{ (n-k)_{\ell}}{n^{\ell}} \left(1-\frac{\lambda}{n}\right)^{n-k-\ell}  e^{-\lambda} = \exp\big( O({k \ell}/{n}) + O({k + \ell}/{n})\big) . 
\end{equation}
Since $\sum_{u \in \tau \cup \tilde{\tau}} k_{u}\ell_{u} \le 2 \Gamma_{n} \sum_{u \in \tau \cup \tilde{\tau}}  \ell_{u} \le 4 \Gamma_{n}^2=o(n)$,  we get 
\begin{equation}
  \frac{\P \left( (\mathcal{S}^{\mathtt{BFS}}_{\le r_{n}}(1)  ,  \mathcal{S}^{\mathtt{BFS}}_{\le r_{n}}(n)= (\tau, \tilde{\tau}) )   \right) }{\P \left( \mathcal{T} = \tau \right)  \P \left( \mathcal{T} = \tilde{\tau} \right) }= \exp \Big( O( {\Gamma_{n}^2}/{n}) \Big). 
\end{equation}
In addition,  conditionally on $(\mathcal{S}^{\mathtt{BFS}}_{\le r_{n}}(1) , \mathcal{S}^{\mathtt{BFS}}_{\le r_{n}}(n))=( \tau, \tilde{\tau})$, the event $ \{\mathcal{S}_{\le r_{n}}(1)  = \mathcal{S}^{\mathtt{BFS}}_{\le r_{n}}(1)\} \cap \{ \mathcal{S}_{\le r_{n}}(n)= \mathcal{S}^{\mathtt{BFS}}_{\le r_{n}}(1) \} \cap \{  \mathcal{S}_{\le r_{n}+1}(1) \cap \mathcal{S}_{\le r_{n}}(n) = \emptyset \} $ occurs if and only if  none of the unexamined edges between $|\tau|+|\tilde{\tau}|$ vertices found is also present, which is of probability at most $\exp \Big( O( {\Gamma_{n}^2}/{n}) \Big).$ Then the desired result follows.
\end{proof}
 
Consequently, the weighted Erd\H{o}s--R\'{e}nyi graph converges locally to a weighted PGW $(\lambda)$ tree, more commonly known as a branching random walk (BRW). Precisely,
given a $\mathrm{PGW}(\lambda)$ tree $\mathcal{T}$ with root $\varrho$, assign to each edge $e$ an i.i.d. weight $X(e)$ with common distribution $X$. For each $u \in \mathcal{T}$, let $V(u)$ denote the sum of weights on the unique path from $\varrho$ to $u$; and let $|u|$ denote denote its graph distance from the root. Define 
\begin{equation}
  W^{\mathtt{BRW}}_{r_{n}} \coloneq  \sum_{|u|=r_{n}}  e^{-\alpha V(u)} .  
\end{equation}
Then   $( W^{\mathtt{BRW}}_{r_{n}})_{r_{n} \ge 0}$
forms a non-negative martingale, and   converges almost surely to a limit:
\begin{equation}\label{eq-Biggins-martingale-limit}
  W = W^{\mathtt{BRW}}_{\infty} \coloneq  \lim_{n \to \infty}   W^{\mathtt{BRW}}_{r_{n}} \ \text{ a.s. }
\end{equation}

Let $(\tilde{W}^{\mathtt{BRW}}_{r_{n}}, \tilde{W})$ be an independent copy of  $({W}^{\mathtt{BRW}}_{r_{n}}, {W})$.

\begin{corollary}\label{lem-W-r_{n}-W}
As $n \to \infty$, $\P( \mathsf{G}_{n}^{1} ) \to 1$. As a consequence,  
  \begin{equation}\label{eq-w-conv-d}
    \lim_{n \to \infty}  ({W}_{r_{n}}, \tilde{W}_{r_{n}}) =  (W,\tilde{W}) \text{ in distribution. }
  \end{equation}  
In particular, the sequence $  ({W}_{r_{n}}, \tilde{W}_{r_{n}})_{n \ge 1}$ is tight. Moreover, for any $\lambda>1$, with 
\begin{equation}\label{eq-w-b-conv}
  \lim_{n \to \infty}   \ind{\mathsf{G}^{1}_n}\sum_{u \in  
  \mathcal{S}_{r_{n}}(1)} e^{-\alpha X( [1, u])} \ind{ |X( [1, u])|> 2 |\psi_{X}'(\alpha) | r_{n} } = 0  \text{ in probability.}
\end{equation}
\end{corollary}

\begin{proof}
  The  assertion $\P( \mathsf{G}_{n}^{1} ) \to 1$ follows from the Lemma \ref{lem-local-weak-convergence},  combined with   the fact   $\E|\{u \in \mathcal{T}: |u|=r_{n}\}|=\lambda^{r_{n}}$. For the second assertion \eqref{eq-w-conv-d}, observe that for any rooted ordered trees $\tau, \tilde{\tau}$
  \begin{align}
    & \E[ f ({W}_{r_{n}}, \tilde{W}_{r_{n}}) \mid \mathcal{S}_{\le r_{n}}(1)  = \tau ,  \mathcal{S}_{\le r_{n}}(n)= \tilde{\tau} ,   \mathcal{S}_{\le r_{n}+1}(1) \cap \mathcal{S}_{\le r_{n}}(n) =\emptyset ]  \\
    & =   \E[ f ({W}^{\mathtt{BRW}}_{r_{n}}, \tilde{W}^{\mathtt{BRW}}_{r_{n}}) \mid \mathcal{T}  = \tau ,   \tilde{ \mathcal{T} }= \tilde{\tau}  ].
  \end{align}  
Then it follows from the law of total probability and Lemma \ref{lem-local-weak-convergence}. For the last assertion, recalling the definition of $(\widehat{S}_n)_{n\ge0}$ from the proof of Lemma~\ref{lem-renewal-function}, we have 
\begin{equation*}
\mathbb{E} \left[ \sum_{|u|=r_{n}}  e^{-\alpha V(u)}  \ind{ |V(u)|> 2 |\psi_{X}'(\alpha) | r_{n} } \right]=  \mathbb{P} \left( |\hat{S}_{r_n}|> 2 |\psi_{X}'(\alpha) | r_{n} \right).   
\end{equation*}
This goes to 0 as $n\to\infty$, because $(\widehat{S}_n)_{n\ge0}$ is a random walk with $\E[\widehat{S}_1] = |\psi_X'(\alpha)| > 0$ by \eqref{eq:exp_hatS1}.
This completes the proof.
\end{proof}

We will later need the following:
\begin{lemma}[Biggins, see e.g.~Theorem 3.2 in Shi~\cite{ShiLectureNotes}]
\label{lem:W equals 0}
The events $\{W=0\}$ and $\{|\mathcal{T}|<\infty\}$ coincide almost surely.
\end{lemma}

\subsection{Conditioned first moment asymptotic}  
Recall that $\lambda_{n}(R)\coloneq \E [ \mathcal{E}_{n}(R) \mid \mathcal{S}_{\le r_{n}}(1,n) ]$ for $R \subset \mathbb{R}^2$. In this subsection we show the  asymptotic behavior of $\lambda_{n}(R)$.

\begin{lemma}\label{lem-cond-set}
We have for every $R\in\mathcal R$,
  \begin{equation}
 \lim_{n \to \infty}  \left|  \lambda_{n}(R) -  W_{r_{n}} \tilde{W}_{r_{n}}  \Lambda(R) \right|  \ind{\mathsf{G}^{1}_n}  =0  \ \text{ in probability}. 
  \end{equation}
% Consequently,  we have 
% \begin{equation}
%   \lim_{n \to \infty} \E \left[    e^{- \lambda_{n}(R) }  \ind{\mathsf{G}^{1}_n}  \right] =   \E \left[    e^{- W  \tilde{ W}  \Lambda(R) } \right] .  
% \end{equation}
\end{lemma}

Before presenting the proof, we  first introduce some necessary notation.
Since on    $\mathsf{G}^{1}_n$, $ \mathcal{S}_{\le r_{n}}(1, n) $  is a forest of two trees,  whose leaves are, respectively, $\mathcal S_{r_n}(1)$ and $\mathcal S_{r_n}(n)$,
any path $p$ from vertex $1$ and to vertex $n$ in $\mathcal{G}_{n}$ can be viewed as a path from some vertex  $u \in \mathcal{S}_{r_{n}}(1)$ to some vertex $v \in \mathcal{S}_{r_{n}}(n)$. Precisely, whenever $\mathsf{G}^{1}_{n}$ occurs, we can decompose $\mathcal{P}(1,n)$ as follows:
\begin{equation}\label{def-B-1-n}
\mathcal{P}(1,n) = \bigcup_{\ell \ge 2} \bigcup_{u \in \mathcal{S}_{r_{n}}(1)}    \bigcup_{v \in \mathcal{S}_{r_{n}}(n)}  \, \mathcal{P}_{u,v}(\ell; \mathcal{S}_{\le r_{n}}(1,n)) ,   
\end{equation} 
 where for each $u \in \mathcal{S}_{r_{n}}(1)$ and $v \in \mathcal{S}_{r_{n}}(n)$, 
 \begin{align}
  \mathcal{P}_{u,v}(\ell; \mathcal{S}_{\le r_{n}}(1,n)) &\coloneq  \Big\{  [1,v_0]_{\mathcal{S}_{\le r_{n}}(1)} \oplus (v_{i})_{i=1}^{\ell-1} \oplus [v_{\ell},n]_{\mathcal{S}_{\le r_{n}}(n)} : \\
 & \qquad v_{0}=u, v_{\ell}=v,  v_{i} \in [n] \setminus \mathcal{S}_{\le r_{n}}(1,n), v_{i} \neq v_{j} \text{ for } i \neq j \Big\}. \label{eq-def-b-u-v-l}
 \end{align} 
For brevity, we write $\mathcal{P}_{u,v}(\ell)$ instead of $\mathcal{P}_{u,v}(\ell; \mathcal{S}_{\le r_{n}}(1,n))$ when the context is clear.
With this notation, $\mathcal{E}_{n}(R)$ can be expressed as
\begin{equation}\label{eq-expression-EnR}
 \mathcal{E}_{n}(R) =  \sum_{u \in \mathcal{S}_{r_{n}}(1)}  \sum_{v \in \mathcal{S}_{r_{n}}(n)} \sum_{\ell \ge 2} \,\sum_{p \in \mathcal{P}_{u,v}(\ell)} I(p; R)  \quad \text{ on } \mathsf{G}^{1}_n. 
\end{equation}
\begin{figure}[tbp]
  \includegraphics[width=0.95\textwidth]{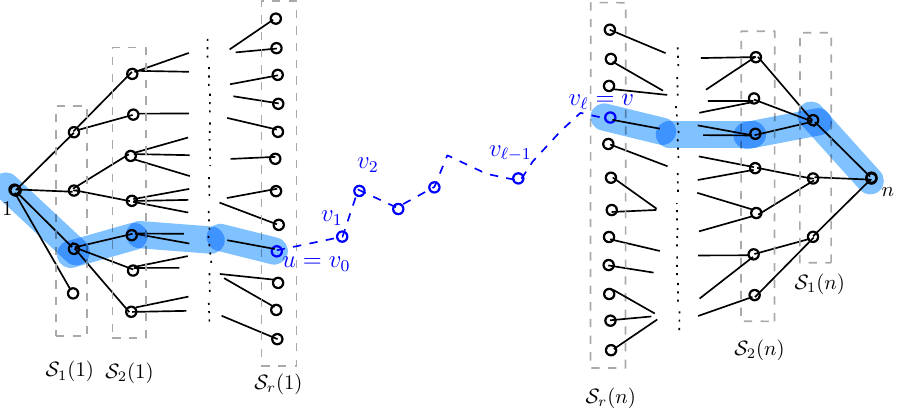}
  \caption{Illustration of a free path $p=[1,u] \oplus (v_{i})_{i=1}^{\ell-1} \oplus [v,n] \in \mathcal{P}_{u,v}(\ell)$.}
\end{figure} 
 
\begin{proof}[Proof of Lemma \ref{lem-cond-set}]
  %[Proof of Lemma \ref{lem-cond-set} admitting Lemma \ref{lem-renewal-function}] 
  As in the proof of Lemma~\ref{lem-set-A}, we can assume that $R = R_{x,h} = (-\infty,x]\times (-\infty,h]$, for $x\in\R^{-\infty}$ and $h\in \R^{-\infty,+\infty}$.
  First, we calculate the expectation $\E[I(p;R_{x,h})]$ for any given   $p \in \mathcal{P}_{u,v}(\ell)$ where $u \in \mathcal{S}_{ r_{n}}(1)$, $v \in \mathcal{S}_{ r_{n}}(n)$, and $\ell \ge 2 $.
  Note that by the choice of $R_{x,h}$, this expectation vanishes whenever $\ell > k_{n}(h)$. For notational convenience, define 
   \begin{equation}\label{eq-def-X-F}
  \begin{aligned}
  \mathsf{X}_{u,v} & \coloneq X([1,u]) + X([v,n]) \ , \ \text{ and } \\
  \mathsf{F}(  \ell,n ,y ) & \coloneq  \P \left(  S_{\ell} \le \frac{1}{\alpha} \ln n + x - y  \right)  \text{ for } y \in \mathbb{R}.
  \end{aligned}
  \end{equation}
Then the conditional expectation is given by 
\begin{equation}\label{eq-exp-I-p-R}
  \E \left[  I(p;R_{x,h}) \mid \mathcal{S}_{\le r_{n}}(1,n)\right] \ind{  \mathsf{G}^{1}_n } =    \left( \frac{\lambda}{n} \right)^{\ell}  \mathsf{F}( \ell, n, \mathsf{X}_{u,v} )  \ind{ \mathsf{G}^{1}_n }  \ind{\ell \le k_{n}(h)}. 
\end{equation} 
Notice that for any  $u \in \mathcal{S}_{ r_{n}}(1)$,  $v \in \mathcal{S}_{ r_{n}}(n)$ and $\ell \ge 2$, 
 \begin{equation}
 |  \mathcal{P}_{u,v}(\ell) | = (n- |\mathcal{S}_{\le r_{n}}(1,n)|)_\ell= [1+o_{n}(1)]   n^{\ell} \text{ on } \mathsf{G}^{1}_{n} .
 \end{equation}
 where  $o_{n}(1) \to 0$ as $n \to \infty$ uniformly in $u,v$ and  $1\le \ell \le \ln^2 n$.   Combined this with \eqref{eq-expression-EnR} and \eqref{eq-exp-I-p-R}, we obtain 
\begin{align}
\lambda_{n}(R_{x,h})  \ind{\mathsf{G}^{1}_n}  & = [ 1+o_{n}(1) ] \sum_{u \in \mathcal{S}_{r_{n}}(1)}  \sum_{v \in \mathcal{S}_{r_{n}}(n)} \frac{1}{n}\sum_{\ell=2}^{k_{n}(h)} \lambda^{\ell} \mathsf{F}( \ell, n, \mathsf{X}_{u,v} )    \ind{\mathsf{G}^{1}_n} .
 \end{align}  

 Let $\mathcal{S}^{\mathrm{g}} $ denote the set of pairs $(u,v) \in \mathcal{S}_{r_{n}}(1) \times  \mathcal{S}_{r_{n}}(n) $ satisfying $|X([1,u])|   \le2  |\psi'_{X}(\alpha)| r_{n}$ and $|X([v,n])|   \le2  |\psi'_{X}(\alpha)| r_{n}$. Set $\mathcal{S}^{\mathrm{b}}  = (\mathcal{S}_{r_{n}}(1) \times \mathcal{S}_{r_{n}}(n)) \setminus  \mathcal{S}^{\mathrm{g}} $.  
 By using part (ii) of Lemma \ref{lem-renewal-function}, since $r_{n}=o(\sqrt{\log n})$ by \eqref{eq:r_n}, we have    $\frac{1}{n}\sum_{\ell=2}^{k_{n}(h)} \lambda^{\ell} \mathsf{F}( \ell, n, \mathsf{X}_{u,v} )  = [1+o_{n}(1)] e^{- \alpha \mathsf{X}_{u,v}} \Lambda(R_{x,h})$ uniformly in $(u,v) \in \mathcal{S}^{\mathrm{g}}$. While for $(u,v) \in \mathcal{S}^{\mathrm{b}}$,   inequality \eqref{eq-renewal-function-bnd} gives $\frac{1}{n}\sum_{\ell=2}^{k_{n}(h)} \lambda^{\ell} \mathsf{F}( \ell, n, \mathsf{X}_{u,v} )  \le C e^{-\alpha \mathsf{X}_{u,v}} e^{\alpha x}$. Employing  these bounds, we deduce that on  $\mathsf{G}^{1}_{n}$, 
 \begin{equation*}
  \lambda_{n}(R_{x,h}) = [1+o_{n}(1) ]
\sum_{(u,v) \in \mathcal{S}^{\mathrm{g}}}  e^{- \alpha \mathsf{X}_{u,v}} \Lambda(R_{x,h})  + O\Big(  e^{\alpha x} \sum_{(u,v) \in \mathcal{S}^{\mathrm{b}}}   e^{- \alpha \mathsf{X}_{u,v}} \Big). 
 \end{equation*}
Note that $\sum_{(u,v)}  e^{- \alpha \mathsf{X}_{u,v}} \ind{\mathsf{G}^{1}_{n}} = W_{r_{n}}  \tilde{W}_{r_{n}} $. So we can rewrite the previous expression as, on $\mathsf{G}^{1}_{n}$,
\begin{equation*}
  \lambda_{n}(R_{x,h}) = [1+o_{n}(1)]  W_{r_{n}}  \tilde{W}_{r_{n}} \Lambda(R_{x,h})   +  e^{\alpha x} O\Big(  W^{\mathrm{b}}_{r_{n}} \tilde{W} _{r_{n}}  + W_{r_{n}} \tilde{W}^{\mathrm{b}}_{r_{n}}  \Big). 
 \end{equation*} 
 where $W^{\mathrm{b}}_{r_{n}}\coloneq \ind{\mathsf{G}^{1}_{n}} \sum_{u \in \mathcal{S}_{r_{n}} (1 )} e^{- \alpha X([1,u]) } \ind{|X([1,u])|> 2 |\psi'_{X}(\alpha)|r_{n}} $ and $\tilde{W}^{\mathrm{b}}_{r_{n}}$ is defined analogously.
% \begin{align}
%    & W^{\mathrm{b}}_{r_{n}}\coloneq \ind{\mathsf{G}^{1}_{n}} \sum_{u \in \mathcal{S}_{r_{n}} (1 )} e^{- \alpha X([1,u]) } \ind{|X([1,u])|> 2 |\psi'_{X}(\alpha)|r_{n}} ; \\
%     &  \tilde{W}^{\mathrm{b}}_{r_{n}}\coloneq \ind{\mathsf{G}^{1}_{n}} \sum_{v \in \mathcal{S}_{r_{n}} (n )} e^{- \alpha X([v,n]) } \ind{|X([v,n])|> 2 |\psi'_{X}(\alpha)|r_{n}}
% \end{align} 
According to Corollary \ref{lem-W-r_{n}-W}, $W_{r_{n}} , \tilde{W}_{r_{n}} $ are tight and $W_{r_{n}}^{\mathrm{b}} ,\tilde{W}_{r_{n}}^{\mathrm{b}}\to 0$ in probability as $n \to \infty$. Thus the desired result follows.
%  \begin{align}
%   \E \left[  \mathcal{E}_{n}(R_{x,h}) \mid \mathcal{S}_{\le r_{n}}(1,n)\right]  \ind{\mathsf{G}^{1}_n}  & = [ 1+o_{n}(1) ] \sum_{u \in \mathcal{S}_{r_{n}}(1)}  \sum_{v \in \mathcal{S}_{r_{n}}(n)} \frac{1}{n}\sum_{\ell=2}^{k_{n}(h)} \lambda^{\ell} \mathsf{F}( \ell, n, \mathsf{X}_{u,v} )    \ind{\mathsf{G}^{1}_n} \\
%   % & = [ 1+o_{n}(1) ] \frac{1}{\gamma} e^{\alpha x - \alpha \mathsf{X}_{u,v}}  \Phi(h)  \ind{\mathsf{G}^{1}_n} 
%   & = [ 1+o_{n}(1) ] \sum_{u \in \mathcal{S}_{r_{n}}(1)}  \sum_{v \in \mathcal{S}_{r_{n}}(n)}  e^ {- \alpha \mathsf{X}_{u,v}}  \Lambda(R_{x,h}) \ind{\mathsf{G}^{1}_n} \\
%   & =  [ 1+o_{n}(1) ] W_{r_{n}} \tilde{W}_{r_{n}}  \Lambda(R_{x,h}) \ind{\mathsf{G}^{1}_n} .
%  \end{align}
%  Above, in the second equality we have used part (ii) of Lemma \ref{lem-renewal-function} {\color{red} (Pascal: we need an assumption on $X_{u,v}$ for this, such as $|X_{u,v}| \le \Gamma_n-|x|$, maybe we can add something like that to the event $\mathsf G^1_{r_{n},n}$?)}.
%  \heng{Right we do need $|X_{u,v}| \le \Gamma_n-|x|$ to apply Lemma \ref{lem-renewal-function}. and should added it in $\mathsf{G}^{1}_{r,n}$. Haven't done yet.}
%   This proves the first assertion. 
%  The second one then follows directly from the first one and  Corollary \ref{lem-W-r_{n}-W}.  
\end{proof}

\section{Limiting Extremal Process}

\subsection{Good events}\label{sec-good}
In this subsection we do some preparation work for the proof of Theorem \ref{thm1}.  We introduce two additional good events. First, define the constant
\begin{equation}\label{eq-def-s-star}
  s^{*} \coloneq  - \inf_{t > \alpha} \frac{\psi_{X}(t)+\ln \lambda }{t}  > 0 .
\end{equation}
The positivity of $s^*$ is guaranteed 
by assumption \eqref{def-alpha}.
%and $X$ is non-trivial so that $\Psi_{X}(t)+\ln \lambda <0$ for any $\alpha<t<\kappa$. 
We then define the event
\begin{equation}
   \mathsf{G}^{2}_{n} \coloneq   \Big\{ \, \forall \text{ path } p \text{ in $\mathcal{G}_{n}$ starting from $1$ or $n$, if } H(p)  \ge r_{n}, \text{ then }  X(p) \ge \frac{ s^{*} }{2} H(p)  \Big\} . 
\end{equation} 
% In particular, on $\mathsf{G}^{1}_n \cap \mathsf{G}^{2}_{n}  $ we have 
% \begin{equation}\label{eq-bound-W-r_{n}}
%   0 \le W_{r_{n}} + \tilde{W}_r \le (2\lambda)^{r_{n}} e^{-\alpha s^{*} (r_{n}-1)/2 } \le (2 \lambda)^r_{n} .   
% \end{equation}  
% Finally in the case where $\kappa <\infty$, we fix an  $\epsilon_0>0$ so that $\frac{\alpha }{\kappa}(1+\epsilon_0)<1$. Define  
% \begin{align} \mathsf{G}^{3}_{n}  &\coloneq   \Big\{  \, \forall \text{ path } p \text{ in $\mathcal{G}_{n}$ },  
%   \,  X(p) \ge - \frac{1+\epsilon_0}{\kappa} \ln n      \Big\}  
% \end{align}  
Finally, with $\alpha'>\alpha$ as in Lemma~\ref{lem-renewal-function}, define
\begin{align} \mathsf{G}^{3}_{n}  &\coloneq   \Big\{  \, \forall \text{ path } p \text{ in $\mathcal{G}_{n}$ },  
  \,  X(p) > -\frac 1 {\alpha'} \ln n  \Big\}  
\end{align}  
and set 
\begin{equation}
    \mathsf{G}_{n} \coloneq   \mathsf{G}^{1}_{n} \cap  \mathsf{G}^{2}_{n} \cap  \mathsf{G}^{3}_{n} . 
\end{equation}   

\begin{lemma}\label{lem-good-event} The good event $ \mathsf{G}_{n}$ occurs with high probability. Precisely, we have 
  \begin{equation}
    \limsup_{n \to \infty} \P( \mathsf{G}_{n}^{c} ) = 0. 
  \end{equation}
\end{lemma}

\begin{proof} In Corollary \ref{lem-W-r_{n}-W} we have proved $ \P((\mathsf{G}_{n}^{1})^{c} )\to 0$. Now let us prove  $\P( (\mathsf{G}_{n}^{2})^{c} )\to 0$. By applying the union bound, we get 
  \begin{align}
    \P(  (\mathsf{G}_{n}^{2})^{c} ) & \le 2\sum_{j =2}^{n} \sum_{\ell \ge r_{n} } \sum_{p \in \mathcal{P}_{\mathrm{f}}(1,j), H(p)=\ell} \P( p \in \mathcal{P}(1,j) ,  X(p) \le  s^{*}\ell /2)  \\ 
    & \le 2 \sum_{\ell \ge r_{n} } n n^{\ell-1} \left(\frac{\lambda}{n} \right)^{\ell} \P( S_{\ell} \le  s^{*}\ell /2)  =  2 \sum_{\ell \ge r_{n} }\lambda^{\ell} \P( S_{\ell} \le  s^{*}\ell /2) .
  \end{align}
Above we have used  $| \{ p \in \mathcal{P}_{\mathrm{f}}(1,j), H(p)=\ell \} | \le n^{\ell-1} $. By employing the Chernoff bound, it follows that for $t>0$
\begin{equation}
  \lambda^{\ell} \P( S_{\ell} \le  s^{*}\ell /2) \le \exp \left( \ell [\psi_{X}(t) + \ln \lambda +s^{*}t/2  ]  \right)
\end{equation}
According to the definition \eqref{eq-def-s-star} of $s^{*}$, there exist  $t_o >\alpha$ and a constant $c_o$ such that  
\begin{equation}
 \psi_{X}(t_o) + \ln \lambda +s^{*}t_o/2    < - c_o <0 .
\end{equation}
Hence the required result $ \P(  (\mathsf{G}_{n}^{2})^{c} ) \le 2 \sum_{\ell \ge r_{n} } e^{- c_o \ell} =O( e^{-c_o r_{n}}) =o(1)$ follows. 

In order to bound  $ \P((\mathsf{G}_{n}^{3})^{c})$, by applying the union bound again, we derive 
\begin{align}
    \P(  (\mathsf{G}_{n}^{3})^{c} ) & \le  \sum_{i,j} \sum_{\ell \ge 1 } \sum_{p \in \mathcal{P}_{\mathrm{f}}(i,j), H(p)=\ell} \P\left( p \in \mathcal{P}(1,j) ,  X(p) \le -\frac 1 {\alpha'} \ln n\right)  \\ 
    & \le   n \sum_{\ell \ge 1}\lambda^{\ell} \P\left( S_{\ell} \le -\frac 1 {\alpha'} \ln n  \right) = nV\left(-\frac 1 {\alpha'} \ln n\right).
  \end{align} 
Part (i) of Lemma \ref{eq-renewal-function-bnd} now implies that $\P(  (\mathsf{G}_{n}^{3})^{c} )\to 0$ as $n\to\infty$. This completes  the proof. 
\end{proof}

\subsection{Convergence of extremal process} In this and the next subsection, we  lay the groundwork for the proof of the main theorem by proving the following convergence result of the conditional avoidance function. Recall that 
\[
\pi_\lambda(k) = e^{-\lambda}\frac{\lambda^k}{k!},\quad k\ge 0.
\]

\begin{lemma}\label{lem-avoidance-con} 
Let $R\in\mathcal R$. Then for every $k\ge 0$,
  \begin{equation}
 \lim_{n \to \infty} \left| \P\left[  \mathsf{G}_{n} ,  \mathcal{E}_{n}(R) = k   \mid \mathcal{S}_{\le r_{n}}(1,n) \right] -  \pi_{\lambda_{n}(R) }(k)    \right| = 0 \text{ in probability.}
  \end{equation}
\end{lemma}

This lemma provides the final piece needed to establish the convergence of the extremal process.
%, as formalized in the following proposition.

% \begin{proposition}\label{prop-extremal-process} The extremal process 
% $(\mathcal{E}_{n}) $ converges weakly to $\mathcal{E}_{\infty}$ in the vague topology as $n \to \infty$. Precisely, for any function $\varphi \in C_{c}(\mathbb{R}^2)$,  
% \begin{equation} 
%  \lim_{n \to \infty} \E \left[  e^{-  \langle \mathcal{E}_{n} , \varphi  \rangle } \right] = \E \left[  e^{-  \langle \mathcal{E}_{\infty} , \varphi  \rangle }  \right] = \E \left[  e^{- W \tilde{W} \int (1- e^{- \varphi}) \dif  \Lambda}  \right].
% \end{equation} 
% \end{proposition}

%\begin{proof}[Proof of Proposition \ref{prop-extremal-process} admitting Lemma \ref{lem-avoidance-con}]
\begin{proof}[Proof of Theorem \ref{thm1} admitting Lemma \ref{lem-avoidance-con}]
Combining Lemmas \ref{lem-cond-set}, \ref{lem-avoidance-con} and \ref{lem-good-event}, we have that $\mathcal{E}_n(R)$ converges in law to $\mathcal{E}_\infty(R)$ for every $R\in\mathcal{R}$. Lemma~\ref{lem:kallenberg} then implies that $\mathcal{E}_n$ converges in law to $\mathcal{E}_\infty$. For the second statement, the convergence just proven directly implies that
\[
\lim_{n\to\infty} \P(\mathcal{E}_n = 0) \le \P(\mathcal{E}_\infty = 0) = \P(W\tilde W = 0).
\]
It therefore remains to show the reverse inequality. By Lemma~\ref{lem:W equals 0}, we have $\P(W\tilde W = 0) = \P(|\mathcal{T}|<\infty,\ |\tilde{\mathcal{T}}|<\infty)$, where $\mathcal T$ and $\tilde{\mathcal T}$ are independent Poi($\lambda$)-BGW trees. But by the local limit result (Lemma~\ref{lem-local-weak-convergence}), we have that 
\[
\P(|\mathcal{T}|<\infty,\ |\tilde{\mathcal{T}}|<\infty) \le \lim_{n\to\infty} \P(\mathcal{P}(1,n) = \emptyset) = \lim_{n\to\infty}\P(\mathcal{E}_n = 0).
\]
This proves the second statement.
%   Since  $\P(\mathcal{E}_{\infty}(\partial R)=0 )=1$ for any rectangle $R$, it is enough to check the  tightness and uniqueness conditions in Lemma \ref{lem-tight+unique}.
%  The first follows from Lemma \ref{lem-set-A}, while the second follows by combining Lemmas \ref{lem-cond-set}, \ref{lem-avoidance-con} and \ref{lem-good-event} with the dominated convergence theorem.
\end{proof}

We now turn to the proof of Lemma \ref{lem-avoidance-con}. Our strategy is to apply Lemma \ref{Stein_method} with 
\begin{align*}
  Y&= Y_n \coloneqq \sum_{\substack{p \in \mathcal{P}(1,n)\\H(p) \le 2\gamma\ln n}} I(p;R) \ , \ \chi= \ind{\mathsf{G}_{n}} \ \text{ and } \ \mathcal{F}= \mathcal{S}_{\le r_{n}}(1,n) .
\end{align*}
Note that we have 
\[
\mathcal{E}_n(R) = Y_n + \mathcal{E}_n(R\cap (\R^{-\infty} \times ((\gamma/\sqrt\beta)\sqrt{\ln n},+\infty])),
\]
and the latter term vanishes in $L^1$ as $n\to\infty$, by Lemma~\ref{lem-set-A}. Hence, setting $\lambda'_{n} \coloneqq \E[Y_n\,|\,\mathcal S_{\le r_{n}}]$, it suffices to prove that
  \begin{equation}
 \lim_{n \to \infty} \left| \P\left[  \mathsf{G}_{n} ,  Y_n = k   \mid \mathcal{S}_{\le r_{n}}(1,n) \right] -  \pi_{\lambda'_{n} }(k)    \right| = 0 \text{ in probability.}
  \end{equation}

Since we will work on the good event $ \mathsf{G}^{1}_n$ which is $ \mathcal{S}_{\le r_{n}}(1,n) $-measurable, we may rewrite $Y_n$ as follows:
 \begin{equation}
  Y_n = \sum_{\substack{p \in \mathcal{P}(1,n)\\H(p) \le 2\gamma\ln n}} I(p;R)  \quad \text{ on } \mathsf{G}^{1}_n,
 \end{equation} 
where  $\mathcal{P}(1,n )$ is defined in \eqref{def-B-1-n}.
Moreover for each $p \in \mathcal{P}(1,n ) $,  let us define the (conditioned) dissociating neighborhood of $p$  as 
\begin{align}
  \mathcal{N}_{p} \coloneq  \Big\{  p' \in \mathcal{P}(1,n) : p' \neq p, \  &  H(p') \le 2 \gamma \ln n,   p' \text{ and } p \text{ share at least one   }  \\
  & \text{ (undirected)  edge that is not contained in }  \mathcal{S}_{\le r_{n}} (1,n) \Big\}.
\end{align}
Explicitly, if $p=  [1,v_0] \oplus (v_{i})_{i=1}^{\ell-1} \oplus [v_{\ell},n]$; and $p' = [1,v'_0] \oplus (v'_{i})_{i=1}^{\ell'-1} \oplus [v'_{\ell'},n]$, then $ p' \in \mathcal{N}_{p}$ if and only if there exists $1 \le i \le \ell$ and $1 \le i' \le \ell'$ such that 
\begin{equation}\label{eq-dissociating-nbhd}
\text{ either }  ( v_{i-1} , v_{i} ) = ( v'_{i'-1} , v'_{i'} ) \text{ or } ( v_{i} , v_{i-1} ) = ( v'_{i'-1} , v'_{i'} )  . 
\end{equation} 

From the construction of $\mathcal{G}_{n}$ it follows that conditionally on $\mathcal{S}_{\le r_{n}}(1,n)$, $I(p;R)$ is   independent of $I(p';R), p' \notin \{p\} \cup \mathcal{N}_{p} \eqqcolon \overline{\mathcal{N}}_p$. By employing Lemma \ref{Stein_method}, we obtain that, for any $R \in\mathcal R$, on the event $\mathsf{G}^{1}_n$, for any $k\in \N$,
\begin{align}
 & \left|   \P\left( \mathsf{G}_{n}, \mathcal{E}_{n}(R) = k   \mid \mathcal{S}_{\le r_{n}}(1,n) \right) -  \pi_{\lambda_{n} (R)}(k)  \right| \\
 &\le  \underbrace{\sum_{\substack{p \in \mathcal{P}(1,n)\\H(p) \le 2\gamma\ln n}} \sum_{p' \in \overline{\mathcal{N}}_p} \E \left[ I(p;R) \mid  \mathcal{S}_{\le r_{n}} (1,n) \right]\E \left[ I(p';R) \mid  \mathcal{S}_{\le r_{n}} (1,n) \right] }_{\Sigma_{\eqref{correlation-1}}(n,r_{n};R)}\label{correlation-1} \\
 &\quad + \underbrace{\sum_{\substack{p \in \mathcal{P}(1,n)\\H(p) \le 2\gamma\ln n}} \sum_{p' \in     \mathcal{N}_p} \E \left[ I(p;R)   I(p';R) \ind{\mathsf{G}_{n}} \mid  \mathcal{S}_{\le r_{n}} (1,n) \right]}_{\Sigma_{\eqref{correlation-2}}(n,r_{n};R)} + 2\P[ \mathsf{G}_{n}^{c} \mid  \mathcal{S}_{\le r_{n}} (1,n) ]. \label{correlation-2}
\end{align}
   Thus it suffices to show that $ \Sigma_{\eqref{correlation-1}} (n,r_{n}; R )$ and $  \Sigma_{\eqref{correlation-2}} (n,r_{n}; R ) $ converges to zero in probability as $n \to \infty$.

\subsection{Vanishing of correlation terms} 
\label{sec-covar}

\begin{lemma}\label{lem-correlation-1}
Let $R\in\mathcal{R}$. Then  we have 
   \begin{equation}
    \lim_{n \to \infty} \Sigma_{\eqref{correlation-1}} (n,r_{n}; R ) \ind{\mathsf{G}_{n}} = 0 \text{ in probability.}
   \end{equation}
\end{lemma}

\begin{proof}
  %[Proof of Lemma \ref{lem-correlation-1}]
  Since $I(p;R)$ is increasing in $R$, it is enough to consider $R$ of the form
  \[
  R = (-\infty,x]\times \R^{-\infty,+\infty},
  \]
  for some $x\in \R$. We assume henceforth that this is the case. Recall from \eqref{eq-exp-I-p-R} that on the good event $\mathsf{G}^{1}_n$ the conditional expectation for $I(p;R)$ is given by  
  \begin{equation}
    \E \left[ I(p;R) \mid  \mathcal{S}_{\le r_{n}} (1,n) \right] \ind{\mathsf{G}^{1}_n} =  \left(\frac{\lambda}{n} \right)^{\ell}  \mathsf{F}(  \ell , n ,   \mathsf{X}_{u,v}  )\ind{\mathsf{G}^{1}_n}, 
  \end{equation}
  where $\mathsf{X}_{u,v} $ and $\mathsf{F}$ are defined in \eqref{eq-def-X-F}, note that $\mathsf{F}$ depends on $x$. This yields that on  $\mathsf{G}^{1}_n$, $  \Sigma_{\eqref{correlation-1}} (n,r_{n}; R) $ is bounded from above by 
  \begin{align}
 \sum_{u, u' \in \mathcal{S}_{r_{n}}(1)} \sum_{v, v'\in \mathcal{S}_{r_{n}}(n)}  \sum_{\ell , \ell' = 2}^{2\gamma\ln n} \sum_{ p \in \mathcal{P}_{u,v}(\ell) } \sum_{ p' \in  \mathcal{P}_{u',v'}(\ell') \cap (\mathcal{N}_{p} \cup \{p\}) } \frac{\lambda^{\ell}}{n^{\ell}}  
 \frac{\lambda^{\ell'}}{n^{\ell'}}  \,  \mathsf{F}(   \ell , n,  \mathsf{X}_{u,v}  ) \mathsf{F}(  \ell' , n, \mathsf{X}_{u',v'} ) .
  \end{align} 
  Next,  given any $p \in \mathcal{P}_{u,v}(\ell)$, we claim that
 the number of $p' \in \mathcal{P}_{u',v'}(\ell')$ in $\overline{\mathcal{N}}_p$ can be bounded from above as follows:   
\begin{equation}\label{eq-nbhd-counting-1}
  | \overline{\mathcal{N}}_p \cap \mathcal{P}_{u',v'}(\ell') |  \le \ell \times \ell' \times   2 \times n^{\ell'-2} . 
\end{equation}
In fact,  according to  \eqref{eq-dissociating-nbhd},  constructing such a path $p'$ involves choosing an index $1 \le i \le \ell$ for $p$ ($\ell$ options), an index $1 \le i' \le \ell'$ for $p'$ ($\ell'$ options), 
the vertex $v'_{i'} \in \{v_{i-1}, v_{i}\}$ ($2$ options),   and the remaining  vertices $(v'_{j})_{j=1}^{\ell-1} \backslash \{v'_{i'-1},v'_{i'}\}$ of $p'$ 
(at most $n^{\ell'-2}$ options---note that one might think there are at most $n^{\ell'-3}$ options, but it might be the case that $i'=1$ so that $v'_0=u'=u=v_0$ has been fixed and similar for $i'=n$). 
  
 Now combining \eqref{eq-nbhd-counting-1} with the trivial bound $|\mathcal{P}_{u,v}(\ell)| \le n^{\ell-1}$, we obtain  
\begin{align}
 &\Sigma_{\eqref{correlation-1}} (n,r_{n}; R) \\ 
 &\le \sum_{u, u' \in \mathcal{S}_{r_{n}}(1)} \sum_{v, v'\in \mathcal{S}_{r_{n}}(n)}  \sum_{\ell , \ell' = 2}^{2\gamma\ln n} 
 n^{\ell-1} \times  2 \ell \ell' \, n^{\ell'-2} \times
  \frac{\lambda^{\ell}}{n^{\ell}}  
 \frac{\lambda^{\ell'}}{n^{\ell'}}  \, \mathsf{F}(   \ell , n,  \mathsf{X}_{u,v}  ) \mathsf{F}(  \ell' , n,  \mathsf{X}_{u',v'} )  \\
  & \le \frac{(2\gamma\ln n)^2}{n} \left[ \sum_{u\in \mathcal{S}_{r_{n}}(1), v \in \mathcal{S}_{r_{n}}(n)} \sum_{\ell \ge 2} \frac{\lambda^{\ell}}{n} \mathsf{F}(   \ell  , n, \mathsf{X}_{u,v} ) \right]^2 . 
\end{align}
By applying part (i) in Lemma \ref{lem-renewal-function}, we have 
\begin{equation}
 \sum_{\ell \ge 2}^{\infty} \frac{\lambda^{\ell}}{n}\mathsf{F}(  \ell , n ,   \mathsf{X}_{u,v}  )  \le \frac{C}{n} e^{\alpha [\frac{1}{\alpha} \ln n + x- \mathsf{X}_{u,v}]} =  C  e^{\alpha x} e^{- \alpha X([1,u])} e^{- \alpha X([v,n])} . 
\end{equation}
Consequently, on the event $\mathsf{G}_{n}^{1}$, we have 
\begin{equation}
\Sigma_{\eqref{correlation-1}} (n,r_{n}; R ) \le C \, \frac{(2\gamma\ln n)^2}{n} e^{\alpha x}  \, W_{r_{n}}^2 (\tilde{W}_{r_{n}})^2   \xrightarrow{n \to \infty} 0 
\end{equation} 
in probability. 
Above we have used the $W_{r_{n}} , \tilde{W}_{r_{n}}$ are tight (Corollary \ref{lem-W-r_{n}-W}).
% Furthermore, on $ \mathsf{G}^{2}_{n}$ we have    $W_{r_{n}}+\tilde{W}_{r_{n}}   \le   (2 \lambda)^r_{n}  $; and hence on $\mathsf{G}_{n} $, it follows that 
% \begin{equation}
%   \Sigma_{\eqref{correlation-1}} (n,r_{n}; R )  \le  C \, \frac{(2\gamma\ln n)^2}{n} e^{\alpha x} \,  (2 \lambda)^{4 r_{n}} \, \xrightarrow{n \to \infty} 0. 
% \end{equation}
Combining with Lemma \ref{lem-good-event}, this completes the proof. 
\end{proof}

\begin{lemma}\label{lem-correlation-2}
Let $r_{n} \ge 1$ and $R\in\mathcal R$. Then  we have   
   \begin{equation}
    \lim_{n \to \infty} \Sigma_{\eqref{correlation-2}} (n,r_{n}; R )\ind{\mathsf{G}_{n}} = 0 \text{ in probability.}
   \end{equation} 
\end{lemma}

 \begin{figure}[htbp]
 \includegraphics[width=\textwidth]{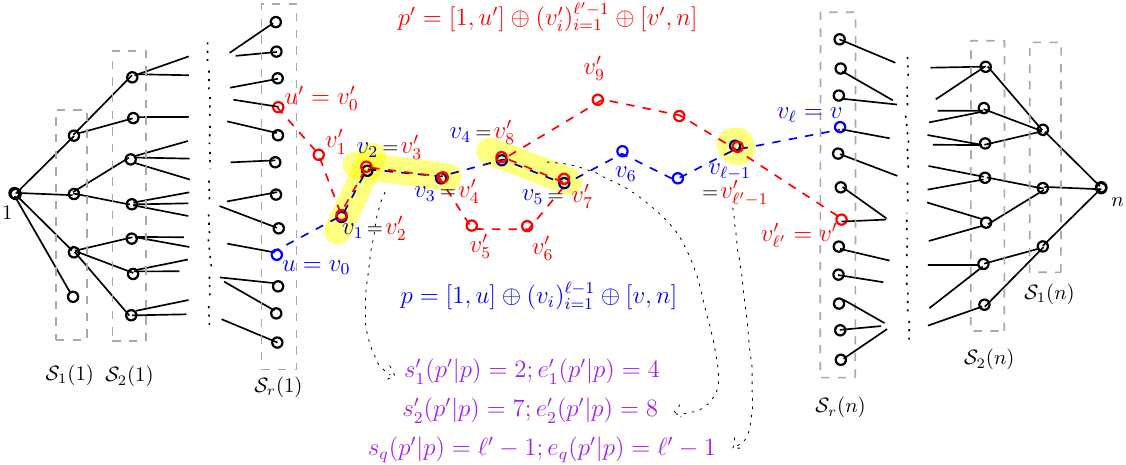}
  \caption{Illustration of $s_{i}'(p'|p)$ and $e'_{i}(p'|p)$.}
\end{figure}

\begin{proof}
  %[Proof of Lemma \ref{lem-correlation-2}] 
As in the proof of Lemma~\ref{lem-correlation-1}, we can assume that $R$ is of the form
  \[
  R = (-\infty,x]\times \R^{-\infty,+\infty},
  \]
  for some $x\in \R$.
  
  Fix any $p  \in \mathcal{P}_{u,v}(\ell)$ with $u \in \mathcal{S}_{r_{n}}(1)$, $  v \in \mathcal{S}_{r_{n}}(n)$ and $2 \le \ell \le 2\gamma\ln n$. 
For each  path $p' \in \mathcal{P}_{u',v'}(\ell')$   that belongs to $\mathcal{N}_{p}$, and $i \ge 1$, we define $ s'_{i}(p'|p) $, $ e'_{i}(p'|p) $ as the start and end indices of the $i$-th ``journey'' in which it shares its trajectory with $p$.  Precisely, let $p=(v_{i})_{i=0}^{\ell}$ and $p'=(v'_{i})_{i=0}^{\ell'}$ be the corresponding representation of $p$, $p'$ \eqref{eq-def-b-u-v-l}.
We define, for $i \ge 1$,
\begin{align}
  s'_{i}(p'|p) & \coloneq  \inf\left\{  j \ge e'_{i-1} : v'_{j} \in \{ v_{i}: 0 \le i \le \ell \} \right\} , \\
   e'_{i}(p'|p) & \coloneq  \sup \left\{  j \ge  s'_{i} : \text{ the pattern  } (v'_{k})_{k=s'_{1}}^{j} \text{ appears in  }  (v_{i})_{i=0}^{\ell}  \text{ or in }  (v_{\ell-i})_{i=0}^{\ell}  \right\} , 
\end{align} 
with the convention  $e'_{0}(p'|p) \coloneq -1$. Denote by $  \ell_{i}(p'|p) $  the number of shared segments during the $i$-th  journey, that is 
\begin{equation}
  \ell_{i}(p'|p) \coloneq  e'_{i}(p'|p)-s'_{i}(p'|p) \ , \ \text{ for } i \ge 1 .  
\end{equation}  
Define $  Q(p'|p)$  as    the total number of  journeys that $p'$  moves along with (or opposite to) $p$: 
\begin{equation}
  Q(p'|p)  \coloneq   \sup\left\{ q: e'_{q} \le \ell' \right\} . 
\end{equation}
Finally, set $ \bm{\ell} (p'|p)\coloneq  \left(  \ell_{i}(p'|p) \right)_{i=1}^{Q(p'|p)} $ and note that 
\begin{equation} 
   |\bm{\ell} (p'|p)| \coloneq   \|\bm{\ell} (p'|p)\|_{1}= \sum_{i=1}^{ Q(p'|p)} \ell_{i}(p'|p) \geq 1 ,  
\end{equation} 
because $p' \in \mathcal{N}_{p}$ shares at least one (undirected) edge with $p$.

Let $(S^{(i)}_{\bullet})_{i \ge 1}$ be i.i.d. copies of the random walk $S_{\bullet}$.
Using this notation, we have 
\begin{equation}\label{eq-correlation-prob}
  \begin{aligned}
  & \E \left[ I(p;R) I(p';R) \ind{\mathsf{G}_{n}} \mid  \mathcal{S}_{\le r_{n}} (1,n) \right]   \\
   & \leq  \left( \frac{\lambda}{n} \right)^{\ell+ \ell'-   |\bm{\ell} (p'|p)|  } \times
   \E \bigg[  \mathsf{F} \Big( n, \ell- |   \bm{\ell} (p'|p)|, \sum_{i=1}^{Q(p'|p)} S^{(i)}_{\ell_{i} (p'|p)} + \mathsf{X}_{u,v}   \Big)  \\
   & \qquad  \times  \mathsf{F}\Big( n, \ell'- |   \bm{\ell} (p'|p)|,    \sum_{i=1}^{Q(p'|p)} S^{(i)}_{\ell_{i} (p'|p)} + \mathsf{X}_{u',v'}   \Big) \ind{S^{(i)}_{\ell_{i}(p'|p)} \ge - \frac{1}{\alpha'} \ln n , \forall \,  1 \le i \le Q(p'|p)  }  \bigg] . 
  \end{aligned}
\end{equation}
Above,    $S^{(i)}_{\ell_{i}}$ comes from the weight of the $i$-th journey that $p'$ follows the trajectory of $p$; and  
we have an inequality not an equality as many constraints in $\mathsf{G}_{n}$ were omitted. 

The  expectation of $I(p;R) I(p';R)$ above suggests a finer classification of $\mathcal{N}_{p}$. To this end, let  $\mathcal{L} \coloneq \left\{ \bm{\ell} =(\ell_{i})_{i=1}^{q}: q \ge 1; \ell_{i} \ge 0,  |\bm{\ell}  |  \coloneq \sum_{i=1}^{q} \ell_{i} \ge 1 \right\}$. 
For  $\bm{\ell}=(\ell_{i})_{i=1}^{q} \in \mathcal{L}$,  define 
\begin{equation}
  \mathcal{P}_{u',v'}(\ell', \bm{\ell} \mid p ) : = \{ p' \in \mathcal{P}_{u',v'}(\ell') \cap \mathcal{N}_{p}: Q(p'|p)=q; \text{ and } \ell_{i}(p'|p)= \ell_{i}, 1 \le i \le q  \} . 
\end{equation}  
This allows us to express $\mathcal{N}_{p}$ as the following disjoint union: 
\begin{equation}\label{eq-union-4}
\mathcal{N}_{p} =  \bigcup_{u' \in \mathcal{S}_{r_{n}}(1), v' \in \mathcal{S}_{r_{n}}(n)} \, \bigcup_{\ell' \ge 2}^{2 \gamma \ln n} \, 
\bigcup_{ \bm{\ell} \in \mathcal{L}}  \, \mathcal{P}_{u',v'}(\ell', \bm{\ell} \mid p )   . 
\end{equation} 
Assuming $p \in \mathcal{P}_{u,v}(\ell)$, we partition $\mathcal{N}_{p}$ into four disjoint subsets: $\mathcal{N}^{\times , \times}_{p}$, $\mathcal{N}^{\checkmark , \times}_{p}$, $\mathcal{N}^{\times,\checkmark}_{p}$, and $\mathcal{N}^{\checkmark , \checkmark}_{p}$. These subsets are the sub-unions of~\eqref{eq-union-4} defined by conditions on the indices $u'$ and $v'$. The first superscript indicates if $u' = u$ ($\checkmark$) or $u' \neq u$ ($\times$), and the second does so analogously for $v'$ and $v$. 
 
\medskip

\underline{\textit{Case 1.}} According to \eqref{eq-union-4} and the definition of $\mathcal{N}^{\times,\times}_{p}$,  we have 
\begin{align}
   \Sigma_{\eqref{eq-sum-FF}} & \coloneq   \sum_{\substack{p \in \mathcal{P}(1,n)\\H(p) \le 2\gamma\ln n}} \sum_{p' \in     \mathcal{N}_p^{\times,\times} } \E \left[ I(p;R)   I(p';R) \ind{\mathsf{G}_{n}} \mid  \mathcal{S}_{\le r_{n}} (1,n) \right]  \\
   &=  \sum_{\substack{u   \in \mathcal{S}_{r_{n}}(1) \\ v  \in \mathcal{S}_{r_{n}}(n) }}  \, \sum_{\ell =2}^{2 \gamma \ln n}  
   \sum_{p \in \mathcal{P}_{u,v}(\ell)} 
   \sum_{\substack{u \neq u'   \in \mathcal{S}_{r_{n}}(1) \\ v \neq v'  \in \mathcal{S}_{r_{n}}(n) }} 
   \sum_{\ell' =2}^{2 \gamma \ln n}   \,  \, \sum_{q \ge 1}\, 
  \sum_{\substack{\bm{\ell}=(\ell_{i})_{i=1}^{q} \\ |\bm{\ell}| \le \ell \wedge \ell'}} \sum_{p' \in \mathcal{P}_{u',v'} (\ell',  \bm{\ell} \mid  p ) }
  \\
  & \qquad  \qquad  \qquad  \qquad  \E \left[ I(p;R)   I(p';R) \ind{\mathsf{G}_{n}} \mid  \mathcal{S}_{\le r_{n}} (1,n) \right]   \label{eq-sum-FF} .
\end{align}
Given any $p\in \mathcal{P}_{u,v}(\ell)$, 
% let us bound $|\mathcal{N}^{\times,\times}_{p}|$.
we claim that for any $u' \neq u$,  $v' \neq v$, $\ell' \ge 2$ and  $\bm{\ell}=(\ell_{i})_{i=1}^{q} \in \mathcal{L}$, 
 \begin{equation}\label{enumeration-1}
  |  \mathcal{P}_{u',v'} (\ell',  \bm{\ell} \mid  p ) |  \le (\ell'    \times \ell  \times   2 )^{q} 
  \times n^{\ell'-1- \sum_{i=1}^{q}(\ell_{i}+1)} . 
\end{equation}
Moreover $  |  \mathcal{P}_{u',v'} (\ell',  \bm{\ell} \mid  p ) |=0$ if $|\bm{\ell}| > \ell$.
In fact, to construct a path $p' \in   \mathcal{P}_{u',v'} (\ell',  \bm{\ell} \mid  p ) $, 
we first select, for each $1 \le  i \leq q$, the indices $1 \leq s_{i}'= s'_{i}(p'|p) \le \ell'$   (at most $\ell'$ options), 
and then choose the corresponding vertex $v'_{s'_{i}}$ in $\{ v_{i}\}_{i=0}^{\ell}$ (at most $\ell$ options). If   $\ell_{i} \ge 1$; we may then determine $(v'_{s'_{i}+1}, ..., v'_{s'_{i}+\ell_{i}})$ to follow either the same direction as  $p$ or the opposite one (at most two choices).  Finally, we choose the remaining vertices in  $(v'_{j})_{j=1}^{\ell'-1} \backslash ( \cup_{i=1}^{q}\cup_{s'_{i} \le j \le s'_{i}+\ell_{i}}\{ v'_{j}   \}) $, which gives at most $n^{\ell'-1- \sum_{i=1}^{q}(\ell_{i}+1)}$ possibilities, since $s'_{1} \ge 1$ and $e'_{q} \le \ell'-1$.

 Now, combining \eqref{eq-correlation-prob} and \eqref{eq-sum-FF}, together with the bounds $|\mathcal{P}_{u,v}(\ell)| \le n^{\ell-1}$ and   \eqref{enumeration-1}, we obtain
\begin{align}
    &\Sigma_{\eqref{eq-sum-FF}} \le \sum_{\substack{u, u' \in \mathcal{S}_{r_{n}}(1) \\ v,v' \in \mathcal{S}_{r_{n}}(n) }} \, \sum_{2 \le \ell, \ell' \le 2\gamma\ln n} \,  \, \sum_{q \ge 1}\, 
  \sum_{\substack{\bm{\ell}=(\ell_{i})_{i=1}^{q} \\ |\bm{\ell}| \le \ell \wedge \ell'}}
  \left(\frac{\lambda}{n}\right)^{\ell+\ell'- |\bm{\ell}| } n^{\ell-1}  2^q (2\gamma\ln n)^{2q} 
  \times n^{\ell'-1- |\bm{\ell}|-q}   \\ 
 & \quad \times  \E_{X} \bigg[  \mathsf{F} \Big(  \ell- |   \bm{\ell} |, n, \sum_{i=1}^{q} S^{(i)}_{\ell_{i} } + \mathsf{X}_{u,v}   \Big)   
  \mathsf{F}\Big(  \ell- |   \bm{\ell} |,  n,  \sum_{i=1}^{q} S^{(i)}_{\ell_{i} } + \mathsf{X}_{u',v'}   \Big) \ind{S^{(i)}_{\ell_{i}} \ge - \frac{1}{\alpha'} \ln n , \forall \,  1 \le i \le q  }  \bigg]  . 
\end{align} 
Notice that the part (i) of  Lemma \ref{lem-renewal-function} yields
\begin{equation}
  \sum_{\ell \ge |\bm{\ell}|} \frac{\lambda^{\ell-|\bm{\ell}|}}{n} \mathsf{F}\Big( \ell- |   \bm{\ell} |, n,   \sum_{i=1}^{q} S^{(i)}_{\ell_{i} } + \mathsf{X}_{u,v}   \Big) \le C e^{\alpha x} \prod_{i=1}^{q} e^{-\alpha S^{(i)}_{\ell_{i} }} e^{-\alpha \mathsf{X}_{u,v} } . 
\end{equation}
Substituting this into the previous bound for  $\Sigma_{\eqref{eq-sum-FF}}$, we get 
\begin{align}
 \Sigma_{\eqref{eq-sum-FF}}  
  &\le \sum_{\substack{u, u' \in \mathcal{S}_{r_{n}}(1) \\ v,v' \in \mathcal{S}_{r_{n}}(n) }} \,   \,  \, \sum_{q \ge 1}\, 
  \sum_{ \bm{\ell}=(\ell_{i})_{i=1}^{q}  }
 \lambda^{|\bm{\ell}|}  \left( \frac{4\gamma\ln n }{n}  \right)^{q} \\
 & \qquad \quad  \times  C^{2} \E_{X} \bigg[   e^{2\alpha x} \prod_{i=1}^{q} e^{- 2 \alpha S^{(i)}_{\ell_{i} }} e^{-\alpha \mathsf{X}_{u,v} } e^{-\alpha \mathsf{X}_{u',v'} } \ind{S^{(i)}_{\ell_{i}} \ge - \frac{1}{\alpha'} \ln n , \forall \,  1 \le i \le q  }  \bigg]   \\
 &=C^{2} e^{2 \alpha x }  \sum_{\substack{u, u' \in \mathcal{S}_{r_{n}}(1) \\ v,v' \in \mathcal{S}_{r_{n}}(n) }}  e^{-\alpha \mathsf{X}_{u,v} } e^{-\alpha \mathsf{X}_{u',v'} } \sum_{q \ge 1} \bigg( \frac{4\gamma\ln n }{n}  \sum_{j \ge 1} \lambda^{j} \E \Big[ e^{-2 \alpha S_{j} } \ind{S_{j} \ge - \frac{1}{\alpha'} \ln n }\Big] \bigg)^{q} . 
\end{align}
%Notice that  if $\kappa > 2 \alpha$ then $  \lambda  \E [ e^{-2 X } ]< 1$ and  $  \sum_{j \ge 1} \lambda^{j} \E \Big[ e^{-2 \alpha S_{j} } \Big] \le ({1-  \lambda  \E [ e^{-2 X } ] })^{-1}$.
%If   $\alpha< \kappa \le 2 \alpha$,  by 
Using Lemma \ref{lem-renewal-function}, it follows that,
\begin{align}
 &  \sum_{j \ge 1} \lambda^{j}  \E \Big[ e^{-2 \alpha S_{j} } \ind{S_{j} \ge - \frac{1}{\alpha'} \ln n }\Big]   \lesssim \sum_{m \ge -   \frac{1}{\alpha'} \ln n   }    e^{-2 \alpha m}    \sum_{j \ge 1} \lambda^{j} \P \Big( S_{j} \in [m-1,m]\Big) 
  \\
  & \quad \le     \sum_{m \ge   - \frac{1}{\alpha'} \ln n   }    e^{-2 \alpha m} V(  m)  \lesssim \sum_{m \ge   - \frac{1}{\alpha'} \ln n   }    e^{-(2 \alpha-\alpha') m} \lesssim n^{\frac{(2\alpha-\alpha')}{\alpha'}}.
\end{align}
Note that $\frac{(2\alpha-\alpha')}{\alpha'} < 1$, since $\alpha' > \alpha$.
Therefore, we can find a constant $c>0$ such that 
\begin{equation}
   \sum_{q \ge 1} \bigg[ \frac{4\gamma\ln n }{n}  \sum_{j \ge 1} \lambda^{j} \E \Big[ e^{-2 \alpha S_{j} } \ind{S_{j \ge - \frac{1}{\alpha'} \ln n }}\Big] \bigg]^{q} \le n^{-c}
\end{equation}
for  sufficiently large $n$.  
Moreover, note that 
\begin{equation}
    \sum_{\substack{u, u' \in \mathcal{S}_{r_{n}}(1) \\ v,v' \in \mathcal{S}_{r_{n}}(n) }}  e^{-\alpha \mathsf{X}_{u,v} } e^{-\alpha \mathsf{X}_{u',v'} } = W_{r_{n}}^{2} (\tilde{W}_{r_{n}})^2
\end{equation}
We conclude that 
$ \Sigma_{\eqref{eq-sum-FF}}   \le n^{-c}  W_{r_{n}}^{2} (\tilde{W}_{r_{n}})^2 .$ Thus as $n \to \infty$, $ \Sigma_{\eqref{eq-sum-FF}} \ind{\mathsf{G}_{n}}$  converges to zero in probability. 

\medskip 

\underline{\textit{Case 2.}}  
Let us bound $|\mathcal{N}^{\checkmark,\times}_{p}|$, where $p\in \mathcal{P}_{u,v}(\ell)$.
(The case for $|\mathcal{N}^{\times, \checkmark}_{p}|$ is similar.) 
Observe that for $p' \in \mathcal{N}^{\checkmark,\times}_{p} $ by definition we have 
 $s'_{1}(p'|p)=0$ and $e'_{q}(p'|p) \le \ell'-1$.
 In this case we claim that  
 \begin{equation}\label{enumeration-2}
  |  \mathcal{P}_{u,v'}^{(p)}(\ell', q, (\ell_{i})_{i=1}^{q} ) |  \le (\ell'    \times \ell  \times   2 )^{q-1} 
  \times n^{\ell'-1- \sum_{i=1}^{q}(\ell_{i}+1) + 1}  
\end{equation}
Comparing with \eqref{enumeration-1}, the  difference here is that there is no freedom to choose  $(v'_{0}, v'_{1},\cdots, v'_{\ell_{1}})$, they must be exactly the same as $(v_{0},v_{1},\cdots,v_{\ell_{1}})$. 
Consequently,   we obtain 
\begin{align}
 &\Sigma_{\eqref{eq-sum-TF}} \coloneq    \sum_{\substack{p \in \mathcal{P}(1,n)\\H(p) \le 2\gamma\ln n}} \sum_{p' \in     \mathcal{N}_p^{\checkmark,\times} } \E \left[ I(p;R)   I(p';R) \ind{\mathsf{G}_{n}} \mid  \mathcal{S}_{\le r_{n}} (1,n) \right] \label{eq-sum-TF} \\
  &\le \sum_{\substack{u, u' \in \mathcal{S}_{r_{n}}(1) \\ v,v' \in \mathcal{S}_{r_{n}}(n) }} \,    \sum_{2 \le \ell, \ell' \le 2\gamma\ln n}   \, \sum_{q \ge 1}\, 
  \sum_{\substack{\bm{\ell}=(\ell_{i})_{i=1}^{q} \\ |\bm{\ell}| \le \ell \wedge \ell'}}
  \left(\frac{\lambda}{n}\right)^{\ell+\ell'- |\bm{\ell}| } n^{\ell-1}  2^{q-1} (2\gamma\ln n)^{2(q-1)} 
  \times n^{\ell'-1- |\bm{\ell}|-q + 1}   \\ 
 & \quad \times  \E_{X} \bigg[  \mathsf{F} \Big( \ell- |   \bm{\ell} |, n, \sum_{i=1}^{q} S^{(i)}_{\ell_{i} } + \mathsf{X}_{u,v}   \Big)    \mathsf{F}\Big( \ell- |   \bm{\ell} |,  n,  \sum_{i=1}^{q} S^{(i)}_{\ell_{i} } + \mathsf{X}_{u',v'}   \Big) \ind{  \mathrm{Con}_{q}^{\checkmark,\times} } \bigg]  , 
\end{align}
where 
\begin{equation}
  \mathrm{Con}_{q}^{\checkmark,\times} \coloneq  \left\{   S^{(1)}_{\ell_{1}} + X([1,u]) \ge \frac{ s^{*}}{2}(r_{n}+\ell_{1}) \right\}  \cap 
  \bigcap_{i=2}^{q} \left\{  S^{(i)}_{\ell_{i}} \ge - \frac{1}{\alpha'} \ln n  \right\} . 
\end{equation}
Following the same argument as in Case 1, we   bound $ \Sigma_{\eqref{eq-sum-TF}}$ from above  by 
\begin{align}
& \sum_{\substack{u, u' \in \mathcal{S}_{r_{n}}(1) \\ v,v' \in \mathcal{S}_{r_{n}}(n) }} \,  \sum_{q \ge 1}\, 
  \sum_{ \bm{\ell}=(\ell_{i})_{i=1}^{q}  }
 \lambda^{|\bm{\ell}|}  \left( \frac{4\gamma\ln n }{n}  \right)^{q-1} 
 C^{2} \E_{X} \bigg[   e^{2\alpha x} \prod_{i=1}^{q} e^{- 2 \alpha S^{(i)}_{\ell_{i} }} e^{-\alpha \mathsf{X}_{u,v} } e^{-\alpha \mathsf{X}_{u,v'} }  \ind{  \mathrm{Con}_{q}^{\checkmark,\times} } \bigg]\\
 & =C^{2} e^{2 \alpha x }  \sum_{\substack{u \in \mathcal{S}_{r_{n}}(1) \\ v,v' \in \mathcal{S}_{r_{n}}(n) }}  e^{-\alpha X([v,n]) -\alpha X(v',n) } \
 \sum_{j\ge 1} \lambda^{j} \E_{X} \left[  e^{-2\alpha [X([1,u]) + S_{j}]} \ind{  S^{(1)}_{\ell_{1}} + X([1,u]) \ge \frac{ s^{*}}{2}(r_{n}+\ell_{1})  } \right]  \\
 & \qquad  \qquad \times  \sum_{q \ge 1} 
 \bigg[ \frac{4\gamma\ln n }{n}  \sum_{j \ge 1} \lambda^{j} \E \Big[ e^{-2 \alpha S_{j} } \ind{S_{j} \ge - \frac{1}{\alpha'} \ln n }\Big] \bigg]^{q-1}.  
\end{align}
By using the fact $\lambda^{j} \E[e^{- \alpha S_{j}}]=1$, we have 
\begin{align}
 & \sum_{j \ge 1} \lambda^{j} \E_{X} \left[  e^{-2\alpha [X([1,u]) + S_{j}]} \ind{  S_{j} + X([1,u]) \ge \frac{ s^{*}}{2}(r_{n}+ j)  } \right] \\
&  \le  \sum_{j \ge 1} e^{-\alpha s^{*} (r_{n}+ j)} \lambda^{j} \E_{X} \left[  e^{-\alpha [X([1,u]) + S_{j}]}  \right] \lesssim  e^{- \alpha s^{*} r_{n}  } e^{ - \alpha X([1,u]) } .
\end{align}
 Simplifying the expression for the upper bound, we finally conclude 
\begin{equation}
  \Sigma_{\eqref{eq-sum-TF}} \lesssim  e^{2 \alpha x} e^{- \alpha s^{*} r_{n} } W_{r_{n}} (\tilde{W}_{r_{n}})^2 .
\end{equation}
Since  $(W_{r_{n}},\tilde{W}_r)$ is tight by Corollary \ref{lem-W-r_{n}-W}, it follows that as $n \to \infty$, $ \Sigma_{\eqref{eq-sum-TF}} \ind{\mathsf{G}_{n}}$  converges to zero in probability. Moreover the same argument yields
\begin{equation}
  \Sigma_{\eqref{eq-sum-FT}} \coloneq    \sum_{\substack{p \in \mathcal{P}(1,n)\\H(p) \le 2\gamma\ln n}} \sum_{p' \in     \mathcal{N}_p^{\times,\checkmark} } \E \left[ I(p;R)   I(p';R) \ind{\mathsf{G}_{n}} \mid  \mathcal{S}_{\le r_{n}} (1,n) \right] \xrightarrow[n \to \infty]{\text{in prob}} 0  \text{ on } \mathsf{G}_{n}  .\label{eq-sum-FT}
\end{equation}

\medskip 

\underline{\textit{Case 3.}}  
Let us bound $|\mathcal{N}^{\checkmark,\checkmark}_{p}|$.  Observe that for any $p' \in \mathcal{N}^{\checkmark,\checkmark}_{p}$, there must hold   
$Q(p'|p) \ge 2$, otherwise $p'=p$ which is absurd since $p \notin \mathcal{N}_{p}$. We claim that for $\bm{\ell}=(\ell_{i})_{i=1}^{q} \in \mathcal{L} $ with $q \ge 2$, 
 \begin{equation}\label{enumeration-3}
  |  \mathcal{P}_{u,v}(\ell', \bm{\ell} | p )  |  \le (\ell'    \times \ell  \times   2 )^{q-2} 
  \times n^{\ell'-1-\sum_{i=1}^{q}(\ell_{i}+1) + 2} 
\end{equation}
Comparing with \eqref{enumeration-1} and  \eqref{enumeration-2}, the  differences here are that now $s'_{1}(p'|p)=0$ and $e'_{q}(p'|p) = \ell'$; and 
there is no freedom to choose  $(v'_{0}, v'_{1},v'_{\ell_{1}})$, and   $(v'_{s'_{q}},..., v'_{\ell'})$: they must be exactly the same as $(v_{0},v_{1},\cdots,v_{\ell_{1}})$ and $(v_{\ell-\ell_{q}},...\cdots, v_{\ell})$  by definition. 
Consequently,   we obtain 
\begin{align}
 &\Sigma_{\eqref{eq-sum-TT}} \coloneq   \sum_{\substack{p \in \mathcal{P}(1,n)\\H(p) \le 2\gamma\ln n}} \sum_{p' \in     \mathcal{N}_p^{\checkmark,\checkmark} } \E \left[ I(p;R)   I(p';R) \ind{\mathsf{G}_{n}} \mid  \mathcal{S}_{\le r_{n}} (1,n) \right] \label{eq-sum-TT} \\ 
  &\le \sum_{\substack{u \in \mathcal{S}_{r_{n}}(1) \\ v \in \mathcal{S}_{r_{n}}(n) }}  \, \sum_{2 \le \ell, \ell' \le 2\gamma\ln n}   \, \sum_{q \ge 2}\, 
  \sum_{\substack{\bm{\ell}=(\ell_{i})_{i=1}^{q} \\ |\bm{\ell}| \le \ell \wedge \ell'}}
  \left(\frac{\lambda}{n}\right)^{\ell+\ell'- |\bm{\ell}| } n^{\ell-1}  2^{q-2} (2\gamma\ln n)^{2(q-2)} 
  \times n^{\ell'-1- |\bm{\ell}|-q + 2}   \\ 
 & \quad \times  \E \bigg[  \mathsf{F} \Big(   \ell- |   \bm{\ell} |, n, \sum_{i=1}^{q} S^{(i)}_{\ell_{i} } + \mathsf{X}_{u,v}   \Big)    \mathsf{F}\Big(  \ell- |   \bm{\ell} |, n,   \sum_{i=1}^{q} S^{(i)}_{\ell_{i} } + \mathsf{X}_{u',v'}   \Big) \ind{  \mathrm{Con}_{q}^{\checkmark,\checkmark} } \bigg]  . 
\end{align} 
where $
  \mathrm{Con}_{q}^{\checkmark,\checkmark}$ is defined to be the intersection of the events $ \{   S^{(1)}_{\ell_{1}} + X([1,u]) \ge \frac{ s^{*}}{2}(r_{n}+\ell_{1})\} $,  $
  \cap_{i=2}^{q-1} \{  S^{(i)}_{\ell_{i}} \ge - \frac{1}{\alpha'} \ln n  \} $ and $ \{   S^{(1)}_{\ell_{q}} + X([v,n]) \ge \frac{ s^{*}}{2}(r_{n}+\ell_{q}) \} $. 
Following the previous calculation  it can be verified that
\begin{equation}
  \Sigma_{\eqref{eq-sum-TT}} \lesssim   e^{2 \alpha x} e^{-2  \alpha s^{*} r_{n} } W_{r_{n}} \tilde{W}_{r_{n}} .
\end{equation}
Hence, as $n \to \infty$, $ \Sigma_{\eqref{eq-sum-TT}} \ind{\mathsf{G}_{n}}$  converges to zero in probability. 

\medskip

\underline{\textit{Final step}}. Recall the definition of $\Sigma_{\eqref{correlation-2}}(n,r_{n};R)$, and that below \eqref{eq-union-4},  we have decomposed $\mathcal{N}_{p}$ as the disjoint union of $\mathcal{N}^{\times , \times}_{p}$, $\mathcal{N}^{\checkmark , \times}_{p}$, $\mathcal{N}^{\times,\checkmark}_{p}$, and $\mathcal{N}^{\checkmark , \checkmark}_{p}$. Therefore, we have  
\begin{equation}
  \Sigma_{\eqref{correlation-2}}(n,r_{n};R) \le \Sigma_{\eqref{eq-sum-FF}} +  \Sigma_{\eqref{eq-sum-TF}} +  \Sigma_{\eqref{eq-sum-FT}}+\Sigma_{\eqref{eq-sum-TT}} . 
\end{equation}
The desired result then follows directly from the conclusions established in Cases 1–3. This completes the proof.
\end{proof}

\appendix

\section{Proof of Lemma \ref{Stein_method}}
\label{app1}

\begin{proof}[Proof of Lemma \ref{Stein_method}]
  It suffices to prove for the unconditional measure $\mathbb{P}$ and apply to the conditional probability measure $\mathbb{P}_{\mathcal{F}}$. 

According to the classical Stein lemma (see e.g. \cite[Lemma 4.2]{Ross11}), for each $A \subset \mathbb{N}$ with function $f_{\lambda,A}(n) \coloneq  \ind{n \in A}-  \sum_{k \in A} \frac{\lambda^k}{k!} e^{-\lambda} $, the Stein equation   
\begin{equation}
 \lambda g (n+1)- n g(n) = f_{\lambda,A}(n) \quad \forall  n \ge 0 
\end{equation}
has a unique solution $g_{\lambda,A}(n) = - \frac{(n-1)!}{\lambda^n} \sum_{k=n}^{\infty} \frac{\lambda^k}{k!} f_{\lambda,A} (k)  $ for $n \ge 1$ and $g_{\lambda,A}(0)=1$. Moreover,  according to Lemma \cite[Lemma 4.4]{Ross11}, the function $g_{\lambda,A}$ satisfies 
  \begin{equation}\label{stein-solution}
    \sup_{n \ge 0} \, \left| g_{\lambda,A}(n+1)- g_{\lambda,A}(n) \right| \le \min \left\{ 1, \frac{1}{\lambda} \right\}  \ , \  \sup_{n \ge 0} \, g_{\lambda,A} (n) \le 1 .
  \end{equation} 
Then we have 
  \begin{align}
      & \mathrm{d}_{\mathrm{TV}}(\chi Y, \mathrm{Poi}(\lambda) ) = \sup_{g=g_{\lambda,A} } \Big| \E\left[  \chi Y g(\chi Y) \right] - \lambda\E\left[   g(\chi  Y +1)   \right] \Big|  \\
      & \le
      \sup_{g=g_{\lambda,A}}  \underbrace{\left|  \E\left[ Y g( Y)   \chi \right] -  \lambda \E\left[ g( Y+1)  \chi \right]    \right| }_{\Sigma_{\eqref{eq-diff}} (g)}+    
      \sup_{g=g_{\lambda,A}}  \lambda g(1) \P\left(     \chi =0 \right)     \label{eq-diff}\\
      & =: \sup_{g=g_{\lambda,A}} \Sigma_{\eqref{eq-diff}}(g) +   \P\left(     \chi =0 \right) .
  \end{align}
  Above, we have used $\lambda g(1) \le 1$ by   \eqref{stein-solution}. 
  For brevity let us denote $p_{i}\coloneq \P (X_{i}=1)=\E[X_i]$,  $Y_{i}\coloneq \sum_{j \in N_{i}} X_{j}$ and $Z_{i}\coloneq \sum_{j \notin \bar{N}_{i}} X_{j}$.  Then  on the one hand,  we have 
  \begin{align}
  \E \left[  \lambda g( Y +1)  \chi \right] = \sum_{i\in I } p_{i} \, \E \left[ g \left(  1+ X_{i} +Y_{i}+ Z_{i}  \right)\chi   \right]   .
  \end{align}  
On the other hand, there holds 
   \begin{align}
  & \E \left[ Y g( Y)   \chi \right] =  \sum_{i \in I}   \E \left[ X_{i} g \left(  X_{i} + Y_{i}+ Z_{i} \right)\chi \right]  =\sum_{i \in I}   \E \left[ X_{i} g \left(  1+ Y_{i}+ Z_{i} \right)\chi \right] \\
   &   =\sum_{i \in I}  p_{i}   \E \left[   g \left(1 + Y_{i}+ Z_{i} \right)\chi \right]   + \sum_{i \in I}    \E \left[ (X_{i}-p_{i}   )g \left(1+ Y_{i}+ Z_{i} \right)\chi \right] . 
  \end{align} 
Thus it follows that 
  \begin{align}
\Sigma_{\eqref{eq-diff}} 
&  \le \sum_{i\in I } p_{i} \, \E \left[ |g \left(  1+ X_{i} +Y_{i}+ Z_{i}  \right)-   g \left(1 + Y_{i}+ Z_{i} \right) | \chi   \right]   \\
& \quad + \sum_{i \in I}   \Big| \E \left[ (X_{i}-p_{i}   ) \, | g \left(1+ Y_{i}+ Z_{i} \right) - g \left(1+   Z_{i} \right) |\chi \right] \Big| \\
& \quad + \sum_{i \in I}   \Big| \E \left[ (X_{i}-p_{i}   )  g \left(1+   Z_{i} \right) \chi \right] \Big| .  
  \end{align}
To bound the first sum, we apply \eqref{stein-solution}, which gives  $|g \left(  1+ X_{i} +Y_{i}+ Z_{i}  \right)-   g \left(1 + Y_{i}+ Z_{i} \right) | \leq X_{i}$. By simply dropping the indicator function $\chi$, it follows that the first sum is bounded above by $\sum_{i \in I} p_i^2$.  
 Similarly, for the second one, 
since \eqref{stein-solution} yields  $|g \left(  1+ Y_{i}+ Z_{i}  \right)-   g \left(1  + Z_{i} \right) | \leq Y_{i}$ and $|X_{i}-p_{i}|\leq X_{i}+p_{i}$, it is bounded from above by 
 \begin{equation}
\sum_{i \in I}   \E \left[ (X_{i} + p_{i} ) Y_{i}\chi \right] \le  \sum_{i \in I} \sum_{j \in N_{i}} \left( p_{i} p_{j} +   \E \left[ X_{i}X_{j}  \chi \right] \right) .
 \end{equation}
 For the last sum, observing  that by that conditionally on $\mathcal{G}$, $X_{i}$ and $Z_{i}$ are independent,  we have $ \E \left[ (X_{i}-p_{i}   )  g \left(1+   Z_{i} \right) \right] =\E \left[ X_{i}-p_{i}   \right]  \E \left[   g \left(1+   Z_{i} \right) \right] = 0 . $ Since $|X_{i}-p_{i}| \le 1$ and $ \| g\|_{\infty} \le 1$ by \eqref{stein-solution} we finally get   
\begin{equation}
  \Big| \E \left[ (X_{i}-p_{i}   )  g \left(1+   Z_{i} \right) \ind{\chi=1} \right] \Big| 
   =  \Big| \E \left[ (X_{i}-p_{i}   )  g \left(1+   Z_{i} \right) \ind{\chi=0} \right] \Big|  
  \le   \P(\chi=0). 
\end{equation} 
This completes the proof.
\end{proof}

\section*{Acknowledgement}

We thank Omer Angel for a stimulating discussion.  
This material is based upon work supported by the National Science
Foundation under Grant No.~DMS-1928930, while the authors were in
residence at the Simons Laufer Mathematical Sciences Institute in
Berkeley, California, during the Spring 2025 semester.
Heng Ma  is supported in part   from a Lady Davis Fellowship at the Technion and further acknowledges partial support, during his visit to SLMath, from the National Key R\&D Program of China (No.~2023YFA1010103) and the NSFC Key Program (Project No.~12231002). 
Pascal Maillard further acknowledges partial support from Institut Universitaire de France, the MITI interdisciplinary program 80PRIME GEx-MBB and the ANR MBAP-P (ANR-24-CE40-1833) project.
 
%%%%%%%%%%%%%%%%%%%%%%%%%%%%%%%%%%
%% REFERENCE %%%%%%%%%%%%%%%%%%%%%
%%%%%%%%%%%%%%%%%%%%%%%%%%%%%%%%%%
 
 \bibliographystyle{alpha}
 \bibliography{biblio}

 \end{document}